\documentclass[12pt]{article}
\usepackage{caption}
\usepackage{url}
\usepackage{color}
\usepackage{amssymb}
\usepackage{amsmath}
\usepackage{amsthm}
\usepackage{graphicx}
\usepackage{bbm}

\newtheorem{theorem}{Theorem}[section]
\newtheorem{lemma}[theorem]{Lemma}

\newtheorem{cor}[theorem]{Corollary}
\newtheorem{ex}[theorem]{Example}
\allowdisplaybreaks

\DeclareMathOperator{\MC}{MC}
\DeclareMathOperator{\diam}{diam}
\DeclareMathOperator{\FLAME}{FLAME}

\DeclareMathOperator{\odd}{odd}

\def\qed{\hfill
\ifhmode\unskip\nobreak\fi\quad\ifmmode\Box\else$\Box$\fi\\ }

\begin{document}

\title{Bipartite Communities}
\author{Kelly B. Yancey \thanks{University of Maryland and Institute for Defense Analyses/Center for Computing Sciences (IDA/CCS), kbyancey1@gmail.com}, Matthew P. Yancey \thanks{Institute for Defense Analyses/Center for Computing Sciences (IDA/CCS), mpyancey1@gmail.com}}
\date{\today}

\maketitle

\begin{abstract}

A recent trend in data-mining is to find communities in a graph. Generally speaking, a community in a graph is a vertex set such that the number of edges contained entirely inside the set is ``significantly more than expected.'' 
These communities are then used to describe families of proteins in protein-protein interaction networks, among other applications. 
Community detection is known to be NP-hard; there are several methods to find an approximate solution with rigorous bounds. 

We present a new goal in community detection: to find good bipartite communities. 
A bipartite community is a pair of disjoint vertex sets $S$, $S'$ such that the number of edges with one endpoint in $S$ and the other endpoint in $S'$ is ``significantly more than expected.'' 
We claim that this additional structure is natural to some applications of community detection. 
In fact, using other terminology, they have already been used to study correlation networks, social networks, and two distinct biological networks. 
We will show how the spectral methods for classical community detection can be generalized to finding bipartite communities, and we will prove sharp rigorous bounds for their performance.
Additionally, we will present how the algorithm performs on public-source data sets.
\end{abstract}



\section{Introduction}

A recent trend in data-mining is to find communities in a graph.
Generally speaking, a community in a graph is a vertex set such that the number of edges contained entirely inside the set is ``significantly more than expected.''
These communities are then used to describe cliques in social networks, families of proteins in protein-protein interaction networks, construct groups of similar products in recommendation systems, among other applications.
For a survey on the state of community detection see \cite{F}.
There are multiple measurements that assess how the number of edges contained in a vertex set exceeds what is expected, and each is considered legitimate for a subset of applications.
Finding an optimum set of vertices is \textbf{NP}-hard for most of these measurements, with a few exceptions \cite{TBGGT}.
The measurement that will be investigated in this paper is \emph{conductance}.

Let $G = (V,E)$ be a weighted undirected graph.  
For shorthand, $ij \in E$ will mean $u_iu_j \in E$.
Also, $ij \in E^{<}$ will represent $ij \in E$ and $i < j$.
The adjacency matrix $A$ is the matrix $(w_{ij})$, where $w_{ij}$ is the weight on the edge $ij$ and $w_{ij}=0$ if $ij \notin E$. 
This paper will operate on the assumption that $w_{ij} > 0$ for all edges $ij$, although this assumption is not ubiquitous.
The degree of a vertex $u_j$ is $d(u_j) = \sum_{ij \in E} w_{ij}$, and the degree matrix $D$ is a diagonal matrix with entries $d(u_i)$.
We assume that our graphs have no isolated vertices; equivalently, $d(u_i) > 0$ for all $i$. 
For the rest of this paper, we will assume that our graphs have $n$ vertices and $e$ edges, unless otherwise specified.

The conductance of a subset of vertices $S$, denoted by $\phi_G(S)$, is the sum of the weights on the edges incident with exactly one vertex of $S$ divided by the sum of the degrees of the vertices in $S$.
Typically it is assumed that the sum of the degrees of the vertices in $S$ is at most half the sum of the degrees of all vertices in $G$, as one can alternatively consider the set $\bar{S} = V - S$.
A \emph{stub} is a half edge - for each edge $uv$, there is a stub incident with $u$ and a stub incident with $v$.
Each stub is given a weight equal to the weight on the edge containing said stub.
Let $C(S)$ be the set of stubs incident with $S$ (the assigned or ``colored'' stubs).
Let $B(S)$ be the set of stubs $s$ from an edge $e$ such that $s$ is incident with $S$, but the other stub from $e$ is not incident with $S$ (the ``bad'' stubs).
For a set of edges and stubs $T$, let $\|T\|$ be the sum of the weights on the stubs plus twice the sum of the weights on the edges.
Using this notation, $\phi_G(S) = \|B(S)\|/\|C(S)\|$.

The \emph{combinatorial Laplacian} is $L' = D - A$ and the \emph{normalized Laplacian} is $L = D^{-1/2}L'D^{-1/2}$.
If we define a vector $x:=x(S)$ such that $x_i = 1$ if $u_i \in S$ and $x_i = 0$ otherwise, then the \emph{Rayleigh quotient} of $x$ is the same as the conductance of $S$:
$$ \mathcal{R}_G(x) := \frac{\sum_{ij \in E^{<}} w_{ij}(x_i - x_j)^2}{\sum_i x_i^2 d(u_i)} = \phi_G(S).$$
Hence, minimizing the value of $\frac{\sum_{ij \in E^{<}} w_{ij}(x_i - x_j)^2}{\sum_i x_i^2 d(u_i)}$ over all vectors $x \in \mathbb{R}^n$ is considered a \emph{continuous relaxation} of the problem of finding a good community.
Note that if $y = D^{-1/2}x$, then $\frac{y^TLy}{y^Ty} = \mathcal{R}_G(x)$.
$L$ and $L'$ are positive semidefinite, and $L$ has eigenvectors with eigenvalues $0 = \lambda_1 \leq \cdots \lambda_n \leq 2$.

The most famous method to ``round'' a solution of the continuous relaxation into a good solution to the original discrete problem is the Cheeger Inequality (see \cite{C}).
Let $e$ be an eigenvector of $L$ corresponding to eigenvalue $\lambda$.  
Let $S_{e,t}$ be the vertex set $\{u_i \in V: d(u_i)^{-1/2}e(u_i) < t\}$.
Cheeger's Inequality states that under these conditions there exists $t'$ such that $\phi_G(S_{e,t'}) \leq \sqrt{2\lambda}$.

There have been multiple heuristic attempts to generalize Cheeger's Inequality by using several eigenvectors, see \cite{VM} for a survey of such algorithms.
An approach with theoretical rigor was found very recently by two groups independently : Louis, Raghavendra, Tetali, and Vempala \cite{LRTV} and Lee, Gharan, and Trevisan \cite{LGT}.
They both showed that for any disjoint $k$ communities $(S_i)_{i=1}^k$ the $\max_i \phi_G(S_i) \geq \lambda_k/2$.

\begin{theorem}[\cite{LGT}, \cite{LRTV}]\label{trevisan thm}
There exist disjoint vertex sets $S_1, \ldots, S_k$ such that for each $i$, we have that $\phi_G(S_i) \leq O(\sqrt{\log(k)\lambda_{kC}})$ for some absolute constant $C$.
Furthermore, there exist disjoint sets $S_1, \ldots, S_k$ such that for each $i$, we have that $\phi_G(S_i) \leq O(k^2 \sqrt{\lambda_{k}})$
\end{theorem}

We consider a new goal in community detection: finding good \emph{bipartite communities}.
A bipartite community is a pair of disjoint vertex sets $S, S'$ such that the number of edges with one endpoint in $S$ and the other endpoint in $S'$ is ``significantly more than expected.''
To this end, we will define a measurement of bipartite conductance.
Let $B'(S,S')$ be the set of edges entirely contained in $S$ or entirely contained in $S'$, and let $\tilde{B}(S, S') = B(S \cup S') \cup B'(S,S')$.
The bipartite conductance of $S,S'$ is $\tilde{\phi}_G(S,S') = \|\tilde{B}(S, S')\| / \|C(S \cup S')\|$.
Because $B(S \cup S') \subseteq \tilde{B}(S,S')$, it clearly follows that $\tilde{\phi}_G(S,S') \geq \phi_G(S \cup S')$, so if $S,S'$ is a good bipartite community then $S \cup S'$ is a good community.
Qualitatively, a good bipartite community is a good community with additional structure, and finding a good bipartite community is a refinement of finding a good community.

We claim that this additional structure is natural to some applications of community detection.
In fact, using other terminology, they have already been used to study protein interactions \cite{LLLW} and group-versus-group antagonistic behavior \cite{ZLL,LSPZL} in online social settings (also known as a ``flame war'').
The study of correlation clustering (see the introduction to \cite{PM} for a survey; also studied under the name ``community detection in signed graphs'' \cite{KSLLLA,WYWLZ,G}) is the special case where an edge may represent similarity \emph{or dissimilarity}, and a recent approach by Atay and Liu \cite{AL} involved bipartite communities.
There are many more possible applications: a network of spammers and their targets will display bipartite behavior.
Another application would be to isolate a regional network of airports inside a global graph of air traffic, where the two sets represent major hub airports and small local airports (the assumptions being that small local airports almost exclusively have flights to geographically close hub airports and major hub airports send relatively few flights to other major hub airports that are geographically close).
Finally, we suggest that it is natural to look for a bipartite relationship when examining co-purchasing networks.
In this case, each side of the community would be different brands of the same product - people are unlikely to purchase two versions of the same product in one shopping trip.

The benefit of looking for the additional structure of a bipartite community in these scenarios is that false positives will be weeded out.
For example, an algorithm for classical community detection algorithms is likely to return the set of international airports at the core of the air transportation network as a community instead of regional networks, because the core of international hubs form a stronger ``Rich-Club'' than even the Internet backbone \cite{CFSV}.
Another benefit is the two-sided labels a bipartite community gives to its members.

Kleinberg considered a related problem \cite{K} for directed graphs when he developed the famous \emph{Hyperlink Induced Topic Search} (HITS) algorithm to find results for a web search query.
His algorithm looked to label a subset of webpages as ``Hubs'' or ``Authorities,'' with the only criteria for such labeling being that Hub webpages have many links to Authority webpages.
The HITS algorithm is then spectral clustering using the eigenvectors of $A^TA$.
Kleinberg's algorithm is famous for its strength, but it does have a known issue of reporting popular websites instead of websites that are popular \emph{in reference to the search query}.
This is because the large eigenvectors of an adjacency matrix are dominated by vertices of high degree \cite{HST}, and the normalized Laplacian is known to present results that better match the topology of the graph.

We take a moment here to use one final application as an example that will help distinguish a bipartite community from a \emph{bicluster}.
A bicluster is a classical community during the special case when the underlying graph $G$ is bipartite.
For example, Kluger, Basri, Chang, and Gerstein \cite{KBCG} find biclusters in a bipartite graph that matched genes to different environmental conditions that affect how those genes are expressed.
On the other hand, Bellay et. al. \cite{Minnesota} found bipartite communities in a graph where genes are matched \emph{to each other} when they affect the expression of each other.
To be specific: the rate of growth of yeast colonies is modified by a known rate when one of the genes in the set is deleted; an edge is added between two genes when the observed modification to the rate of growth after both genes are deleted is statistically different from the product of the modifications from each independent gene deletion.
This particular study of gene interaction is called \emph{double mutant combinations}, and bipartite communities are suggested to correspond to redundant pathways \cite{Minnesota}.

We will investigate the existence of bipartite communities in several public source data sets, including the double mutant combination network for yeast cells.
We will also look for bipartite communities in a network of political blogs; our results will match Kleinberg's model for the internet.

We will show that bipartite communities can be found using the largest eigenpairs of $L$.
This is not the first time that the largest eigenpairs of $L$ and $L'$ have been studied.
They are frequently seen as duals to the small eigenpairs of $L$ and $L'$ (see \cite{BJ} and \cite{L}).
They have also been the focus of independent interest because of the related problem of MAX-CUT.
The problem of MAX-CUT is to find a vertex set $S$ such that $\|B(S)\|$ is maximized (and equivalently $\|B'(S, V-S)\|$ is minimized).

Let $\MC = \max_S \|B(S)\|$.
If $\lambda_n'$ is the largest eigenvalue of $L'$, then $\|B(S)\| \leq \lambda_n' \frac{|S|(n-|S|)}{n}$ \cite{MP}.
It follows that $\MC  \leq \frac{\lambda_n' n}{4}$.
Certain strengthenings of this are possible, giving tight results for specific classes of graphs \cite{DP}.
There is a similar proof to show that $\MC \leq e\lambda_n/2$.

One of the most recent results in approximate solutions to MAX-CUT is from Trevisan, who recursively seeks out bipartite communities and returns a set of vertices that is the union of one of the two vertex sets from each bipartite community.
If $\tilde{x}:=\tilde{x}(S,S')$ for some bipartite community $S,S'$ where $\tilde{x}_{i} = 1$ if $u_i \in S$, $\tilde{x}_i = -1$ if $u_i \in S'$, and $\tilde{x}_i = 0$ otherwise, then 
$$ 2 - \mathcal{R}_G(\tilde{x}) = \frac{\|B(S \cup S')\| + 2\|B'(S,S')\|}{\|C(S \cup S')\|} = c_S \tilde{\phi}_G(S,S'),\ c_S \in [1,2]$$
while we had equality for classical communities.
It follows that for any $r$ disjoint bipartite communities $S_i, S_i'$, we have the $\max_i \tilde{\phi}_G(S_i,S_i') \geq 1 - \lambda_{n+1-r}/2$.

\begin{theorem}[Trevisan \cite{T}] \label{one dim algorithm}
Let $e$ be an eigenvector of $L$ corresponding to eigenvalue $\lambda$.  
For $t > 0$, let $S_{e,t}$ be the vertex set $\{u_i \in V: d(u_i)^{-1/2}e(u_i) < -t\}$ and $S_{e,t}'$ be the vertex set $\{u_i \in V: d(u_i)^{-1/2}e(u_i) > t\}$ .
Under these conditions, there exists a $t'$ such that $\tilde{\phi}_G(S_{e,t'}, S_{e,t'}') \leq \sqrt{2(2-\lambda)}$.
\end{theorem}

Liu \cite{L} showed that there exists $k$ disjoint bipartite communities that satisfy $\tilde{\phi}_G(S, S') \leq O\left(k^3 \sqrt{2 - \lambda_{n+1-k}}\right)$.
The main theoretical work of this paper is to strengthen this bound.

\begin{theorem} \label{main}
Fix a value for $k$.
There exists disjoint sets $S_1, S_1', S_2, S_2', \ldots, S_r, S_r'$ such that for any graph $G$ and each $1 \leq i \leq r$,\\
(A) $r = k$ and $\tilde{\phi}_G(S_i, S_i') \leq  \frac{2(8k + 1)(4k -1)}{k+1 - i} \sqrt{\frac{\sum_{1 \leq i \leq k}(2 - \lambda_{n+1 - i})}{k} }$.\\
(B) $r \leq k/2$ and $\tilde{\phi}_G(S_i, S_i') \leq  \frac{10^{1.5}(1280\sqrt{3\ln(200k^2)}+4)k}   {9\left(\frac{k}{2} + 1 - i\right)}     \sqrt{\frac{\sum_{1 \leq i \leq k}(2 - \lambda_{n+1 - i})}{k} }$.
\end{theorem}

\newpage

To summarize the result asymptotically:

\begin{cor} \label{main cor}
There exists a constant $C$ such that for any graph $G$ and value of $k$ there exist disjoint sets $S_1, S_1', S_2, S_2', \ldots, S_r, S_r'$ such that for each $1 \leq i \leq r$,\\
(A) $r = k$ and $\tilde{\phi}_G(S_i, S_i') \leq  C\frac{k^2}{k+1-i} \sqrt{2 - \lambda_{n+1-k}}$. \\
(B) $r = k/4$ and $\tilde{\phi}_G(S_i, S_i') \leq  C\sqrt{\log(k) (2 - \lambda_{n+1-k})}$.
\end{cor}

Liu \cite{L} proved that large unweighted cycles satisfy 
$$0.45 \sqrt{2 - \lambda_{n+1-k}} \leq \min_{S_1, S_2, \ldots, S_k, S_k'} \max_i \tilde{\phi}_{C_N}(S_i, S_i') \leq 1.7 \sqrt{2 - \lambda_{n+1-k}}.$$ 
There exist several examples that show that the $\sqrt{\log(k)}$ term is necessary for Theorem \ref{trevisan thm} (see \cite{LGT} and \cite{LRTV}).
We modify one of those examples to demonstrate the sharpness of Corollary \ref{main cor}.  We call this example the \emph{Bipartite Noisy Hypercube}.

\begin{ex}[Bipartite noisy hypercubes] \label{bipartite noisy hypercubes}
Let $k$ and $c$ be fixed, with $1 \leq c \leq \frac{10k}{22}$, and let $\epsilon = \frac{1}{\log_{2.2}(k/c)}$.
Let $G^{(o)}_{k,c}$ be the weighted complete bipartite graph on $2^k$ vertices, where $V = \{0,1\}^k$, an edge $xy$ exists if and only if $\|x-y\|_1$ is odd, and the weight of edge $xy$ is $\epsilon^{\|x - y\|_1}$.
In $G^{(o)}_{k,c}$ we have that $2 - \lambda_{n - k} \leq 3 \epsilon$ and for any set $T,T' \subset V$ with $|T \cup T'| \leq \frac{c}{k}|V|$ we have that $\tilde{\phi}(T,T') \geq 1/2$.
\end{ex}

The outline for the rest of the paper is as follows:  
in Section 2 we prove Theorem \ref{main}, 
followed by the details of Example \ref{bipartite noisy hypercubes} in Section 3,
and concluding with Section 4 where we present a heuristic algorithm and empirical results on its performance.

\section{Proof of Theorem \ref{main}}
Louis, Raghavendra, Tetali, and Vempala \cite{LRTV} and Lee, Gharan, and Trevisan \cite{LGT} used different approaches to prove Theorem \ref{trevisan thm}.
Both groups considered the $k$ eigenvectors as a mapping into $\mathbb{R}^k$.
The former randomly projected the points in spectral space onto the axes, where each axis forms a candidate community to be calculated using the same procedure as Cheeger's Inequality.
The latter grouped points together in $\mathbb{R}^k$ using a random $\epsilon$-net followed by a test of magnitude for community membership.
Our approach is a hybrid of these arguments: we will partition the points in $\mathbb{R}^k$ randomly, and each part of the partition will be deterministically projected onto an axis where a community will be calculated using the same procedure as Theorem \ref{one dim algorithm}.

\subsection{Definitions and set-up}\label{sub def}
We will be examining the signless normalized Laplacian $\tilde{L} = I + D^{-1/2}AD^{-1/2}$ and the smallest eigenpairs of $\tilde{L}$.
Because $\tilde{L} = 2I - L$, the eigenvalues of $\tilde{L}$ are $\tilde{\lambda}_i = 2 - \lambda_{n+1-i}$, and the eigenvectors of $L$ are the eigenvectors of $\tilde{L}$ in reverse order.

Let $F:V \rightarrow \mathbb{R}^k$ be a map, and let $\|x-y\|$ be the standard Euclidean distance between points $x,y \in \mathbb{R}^k$.
We define the \emph{signless Rayleigh quotient of $F$} to be $\tilde{\mathcal{R}}_G(F) = \frac{\sum_{ij \in E^<} w_{ij}\|F(u_i) + F(u_j)\|^2}{\sum_{i \in V} \|F(u_i)\|^2 d(u_i)}.$
If $f:V\rightarrow\mathbb{R}$ and $e = D^{1/2}f$, then 
$$\tilde{\mathcal{R}}_G(f) =  \frac{e^T \tilde{L} e}{e^T e}.$$
Let $e_1, e_2, \ldots, e_k$ be the eigenvectors of $\tilde{L}$ that correspond to the smallest eigenvalues, and for each $i$, let $e_i = D^{1/2}f_i$.
Because $\tilde{L}$ is symmetric, we may choose our $e_i$ to be orthonormal.
It follows that $\tilde{\mathcal{R}}_G(f_i) = \tilde{\lambda}_i$.
It is also an easy calculation to see that $\sum_j d(u_j) f_i(u_j)^2 = 1$ for all $i$.
We choose $F(u) = (f_1(u), f_2(u), \ldots, f_k(u))$.

For each vertex $u$ with $F(u) \neq 0$, let $F'(u) = F(u)/\|F(u)\|$.
For this type of operation we will modify the \emph{radial projection distance}, which is $d_F(x,y) = \left\| F'(x) - F'(y) \right\|$ when well defined and $d_F(x,y) =1$ when $F(x)$ or $F(y)$ is the origin.
This is the distance function used by Lee, Gharan, and Trevisan to cluster points in spectral space to find subsets of vertices with low conductance.
The radial projection distance can be thought of as an angle-based distance because if 
$\theta$ is the angle between $F(x)$ and $F(y)$, then $d_F(x,y) = 2\sin(\theta/2)$ when $F(x), F(y)$ are not the origin.

However, to find subsets of vertices that have low bipartite conductance, we wish to cluster a vertex $u$ with vertices that map to a point close to $-F(u)$ as well as close to $F(u)$.
For points $x$ and $y$, we define the \emph{mirror radial projection} to be $d_M(x,y) = d_F(x,y)$ when $x^Ty \geq 0$ and $d_M(x,y) = \left\| F'(x) + F'(y) \right\|$ otherwise.
This is equivalent to the distance function on the appropriate projective space.
Let $F_{uv}(u) = F(u)$ and $F_{uv}'(u) = F'(u)$ if $u^Tv > 0$, and $F_{uv}(u) = -F(u)$ and $F_{uv}'(u) = -F'(u)$ otherwise.
As a slight abuse of notation, for vertices $u,v$ we use the shorthand notation $d_M(F(u), F(v)) = d_M(u,v)$, which equals $\|F'(u) - F_{vu}'(v)\|$ when $F(u),F(v) \neq 0$.
If $\theta^*$ is the angle between $F(u)$ and $F_{uv}(v)$, then $d_M(u, v) = 2\sin(\theta^*/2)$.
For fixed vertex $u$ we have that $d_M(w,w') = \|F'_{wu}(w) - F_{w'u}'(w')\|$ behaves like standard Euclidean distance for all pairs of vertices $w,w'$ such that $d_M(u,w),d_M(u,w')\leq 2^{1/2}$.
When not specified, all distance functions are assumed to be $d_M$.
We define a ball for $u \in V$ to be $B_t(u) \subseteq V$ such that $B_t(u) = \{w \in V: d_M(u, w) < t\}$.

For a set of points $S$ and distance function $d$, we write the \emph{diameter} of $S$ as $\diam(S) = \sup_{x,y \in S}d(x, y)$.
For a set of vertices $T$, we define the \emph{volume} to be $\mathcal{V}(T) = \sum_{u \in T}d(u)$ and the \emph{mass} to be $\mathcal{M}(T) = \sum_{u \in T}d(u)\|F(u)\|^2 $.
We will use $P$ to denote a partition of the vertex set, and $P(u)$ to denote the part of the partition that contains vertex $u$.
We say that $F$ is $(\Delta, \eta)$-\emph{spreading} if for every subset of vertices $S$ with diameter less than $\Delta$ has mass at most $\eta \mathcal{M}(V)$.
The support of a map $Q:D\rightarrow \mathbb{R}^k$ is the subset of the domain $D' \subset D$ that is defined by $\|Q(x)\| \neq 0$ if and only if $x \in D'$.

Note that 
\begin{eqnarray*} 
\tilde{\mathcal{R}}_G(F) 	& = &  		\frac{\sum_{ab \in E^<} w_{ab}\|F(u_a) + F(u_b)\|^2}{\sum_{a \in V} d(u_a) \|F(u_a)\|^2 } \\
				& = & 		\frac{\sum_i \sum_{ab \in E^<} w_{ab} |f_i(u_a) + f_i(u_b)|^2}{\sum_i \sum_{a \in V} d(u_a)|f_i(u_a)|^2 } \\
				& = & 		\frac{\sum_i \left(\tilde{\mathcal{R}}_G(f_i) \sum_{a \in V} d(u_a)|f_i(u_a)|^2  \right)}{\sum_i \sum_{a \in V} d(u_a)|f_i(u_a)|^2} \\
				& = & 		\frac{\sum_i \tilde{\lambda}_i}{k}.
\end{eqnarray*}

\subsection{Finding a partition with tightly concentrated parts and balanced mass}\label{sub part}

\begin{lemma}\label{spreading}
If $2^{-1/2} \geq \Delta > 0$, then $F$ is $\left( \Delta, \frac{1}{k(1 - \Delta^2)}\right)$-spreading.
\end{lemma}
\begin{proof}
Let $S$ be a set of points with diameter at most $\Delta$ and $v \in S$.
If $S$ contains a point at the origin and has diameter less than $1$, then $|S| = 1$.
Furthermore, points at the origin have no mass, and thus the lemma is true trivially.
So we may restrict our attention to vertices that $F$ does not map to the origin.
Let $\theta_{vw}$ be the angle between vectors $F'(v)$ and $F(w)$, and let $\theta_{vw}^*$ be the angle between vectors $F'(v)$ and $F_{vw}(w)$.  Observe,

\begin{eqnarray*}
1	& = &	\|F'(v)\|^2	\\
	& = &	\sum_{i\in\{1,\ldots,k\}} f'_i(v)^2 \cdot 1 + \sum_{i\neq j\in\{1,\ldots,k\}} f'_i(v)f'_j(v) \cdot 0 \\
	& = &	\sum_{i\in\{1,\ldots,k\}} f'_i(v)^2 \|e_i\|^2 + \sum_{i\neq j\in\{1,\ldots,k\}} f'_i(v)f'_j(v) (e_i^Te_j) \\
	& = & 	\sum_{i,j\in\{1,\ldots,k\}} (f'_i(v) e_i)^T (f'_j(v) e_j) \\
	& = & 	\left(\sum_{i\in\{1,\ldots,k\}} f'_i(v) e_i\right)^T 			\left(\sum_{j\in\{1,\ldots,k\}} f'_j(v) e_j\right)\\
	& = & 	\left\| \sum_{i\in\{1,\ldots,k\}} f'_i(v) e_i \right\|^2 \\
	& = &	\sum_{w \in V} \left(\sum_{i\in\{1,\ldots,k\}} f'_i(v) e_i(w)\right)^2  \\
	& = & 	\sum_{w \in V} d(w) \left(\sum_{i\in\{1,\ldots,k\}} f'_i(v) f_i(w)\right)^2 \\
	& = & 	\sum_{w \in V} d(w) (F'(v)^T F(w))^2 \\
	& = & 	\sum_{w \in V} d(w) \|F(w)\|^2 \cos^2(\theta_{vw}) \\
	& = & 	\sum_{w \in V} d(w) \|F(w)\|^2 (1 - \sin^2(\theta_{vw}^*)) \\
	& \geq &\sum_{w \in S} d(w) \|F(w)\|^2 (1 - (2\sin(\theta_{vw}^*/2))^2) \\
	& = & \sum_{w \in S} d(w) \|F(w)\|^2 (1 - d_M(v,w)^2) \\
	& \geq & (1 - \Delta^2) \mathcal{M}(S).
\end{eqnarray*}

The statement of the lemma then follows by comparing this to 

\begin{eqnarray*}
 \mathcal{M}(V)				& = &   \sum_{w \in V} d(w) \| F(w)\|^2 \\
					& = &   \sum_{w \in V} d(w) \sum_{i \in \{1, \ldots, k\}} f_i(w)^2 \\
					& = &	\sum_{i \in \{1,\ldots,k\}} \sum_{w \in V} d(w) f_i(w)^2 \\
					& = &   \sum_{i \in \{1,\ldots,k\}} 1 \\
					& = & k .
\end{eqnarray*}
\end{proof}

We will partition our space by greedily assigning new points to a part; suppose that $V' \subset V$ have been assigned to some part. 
Pick a random point $x \in \mathbb{R}^k$, and create a new part equal to $(V - V') \cap \{u:F'(u) \in B_{\Delta/2}(x)\}$.
Repeat until $V' = V$.
Charikar, Chekuri, Goel, Guha, Plotkin \cite{CCGGP} proved that this simple algorithm performs reasonably well.

\begin{lemma}[\cite{CCGGP}]\label{partition happy}
There exists a randomized algorithm to generate a partition $P$ such that each part of the partition has diameter at most $\Delta$ and 
$$ \mathbb{P}[P(u) \neq P(v)] \leq \frac{2 \sqrt{k} d_M(u,v)}{\Delta}.$$
\end{lemma}

Each of our communities will be a subset of a union of parts.
We will produce a lemma that shows that edges $uv$ with $P(u) \neq P(v)$ contribute very little to the term $\sum_{uv \in E^<} w_{uv}\|F(u) + F(v)\|^2$.

\begin{lemma} \label{wooooo}
For any $F:V \rightarrow \mathbb{R}^k$ and $u,v \in V$ such that $F(u),F(v) \neq 0$, we have that $d_M(u, v)\|F(u)\| \leq 2 \|F(u) + F(v)\|$.
\end{lemma}
\begin{proof}
For brevity, let $x = F(u)$ and $y = F_{uv}(v)$. 
Among all vectors $z$ with magnitude $\rho$, the one that minimizes $\|y-z\|$ is $z = \rho\frac{y}{\|y\|}$.
Using this, we see that 
\begin{eqnarray*}
 \|F(u)\|  d_M(u, v)    & = &    \|x\| \left\| \frac{x}{\|x\|} - \frac{y}{\|y\|} \right\| \\
 		& = & \left\| x - \frac{\|x\|}{\|y\|} y \right\|  \\
		& \leq & \left\| y - \frac{\|x\|}{\|y\|} y \right\| + \|x-y\|   \\
		& \leq & 2 \|x - y\|.
\end{eqnarray*}

If $F(u)^T F(v) < 0$, then $\|x-y\| = \|F(u) + F(v)\|$, and the lemma follows.
If $F(u)^T F(v) \geq 0$, then $\|x-y\| = \|F(u) - F(v)\| \leq \|F(u) + F(v)\|$, and the lemma follows. 
\end{proof}

\begin{lemma} \label{I think I solved it}
For any $F:V \rightarrow \mathbb{R}^k$ and $u,v \in V$ such that $F(u),F(v) \neq 0$, we have that 
$$ \sum_{u \in V} \sum_{v \in N(u)}   w_{uv} d_M(u, v)\|F(u)\|^2 \leq  \sqrt{8\tilde{\mathcal{R}}(F)^{-1}}  \sum_{uv \in E^<} w_{uv}\|F(u) + F(v)\|^2 .$$
\end{lemma}

\newpage

\begin{proof}
Apply the Cauchy-Schwartz formula to see that 
\begin{eqnarray*}
\sum_{u \in V} \sum_{v \in N(u)}  w_{uv} d_M(u, v)\|F(u)\|^2
			& \leq & \sum_{u \in V} \sum_{v \in N(u)} w_{uv}2\|F(u) + F(v)\|\|F(u)\| \\
			& \leq & 2 \sqrt{\sum_{u \in V} \sum_{v \in N(u)} w_{uv}\|F(u)\|^2}\cdot\\
			& & \sqrt{\sum_{u \in V} \sum_{v \in N(u)} w_{uv}\|F(u) + F(v)\|^2} \\
			& = &	\sqrt{8\tilde{\mathcal{R}}(F)^{-1}}  \sum_{uv \in E^<} w_{uv}\|F(u) + F(v)\|^2 .
\end{eqnarray*}
\end{proof}

\subsection{The main result}

\begin{theorem} \label{all the work}
If we have a randomized method to generate a partition $P$ with $r$ parts such that $\mathbb{P}[P(u) \neq P(v)] \leq C_1 d_M(u,v)$ and each part has mass at least $C_2\mathcal{M}(V(G))$, 
then there exists vertex sets $S_1, \ldots, S_r, S_1', S_2', \ldots, S_r'$ where $\tilde{\phi}(S_i, S_i') \leq \frac{8C_1+4}{C_2(r - i + 1)} \sqrt{\tilde{\mathcal{R}}(F)}$. 
\end{theorem}

\begin{proof}
Let $\chi$ denote an indicator variable.
Choose a partition $P$ that performs at least as well as the expectation in the sense that 
\begin{equation}\label{pick P}
 \sum_{u \in V}\sum_{v \in N(u)}w_{uv}\chi(P(u) \neq P(v))\|F(u)\|^2 \leq  \sum_{u \in V}\sum_{v \in N(u)}w_{uv}C_1 d_M(u,v)\|F(u)\|^2
\end{equation}

Fix some $i$; we will find the communities $S_i, S_i'$ independently.
We will project $F$ onto one of its coordinates $j^{(i)}$, and use $f_j$ instead of $F$.  
When there is no chance for confusion, we will use $j$ as shorthand for $j^{(i)}$.
If we choose a $j$ at random then the terms $f_j(u)^2$ and $(f_j(u) + f_j(v))^2$ have expectation $\|F(u)\|^2/k$ and $\|F(u) + F(v)\|^2/k$.
Choose a $j$ such that there exists an $\alpha_i$ where
$$ 0 \neq \sum_{u \in P_i} d(u) f_j(u)^2 = \alpha_i \sum_{u \in P_i} d(u) \|F(u)\|^2 $$
and 
\begin{equation}\label{pick j}
 \alpha_i^{-1} \sum_{u \in P_i} \sum_{v \in N(u)} w_{uv} \left(C_1\tilde{\mathcal{R}}(F)^{-1/2} (f_j(u) + f_j(v))^2 + \chi(P(u) \neq P(v))f_j(u)^2\right)
\end{equation}
$$  \leq  \sum_{u \in P_i} \sum_{v \in N(u)} w_{uv} \left(C_1\tilde{\mathcal{R}}(F)^{-1/2} \|F(u) + F(v)\|^2 + \chi(P(u) \neq P(v))\|F(u)\|^2\right)$$
Each $j$ may be chosen independently for each fixed $i$, but there is only one partition $P$, which was chosen before we fixed the value for $i$.
We will then use (\ref{pick j}) by summing across all values for $i$ at once, where the right hand side becomes (after using  Lemma \ref{I think I solved it} and (\ref{pick P}))
$$ C_1\tilde{\mathcal{R}}(F)^{-1/2} 2(1 + \sqrt{2}) \sum_{uv \in E^<} w_{uv} \|F(u) + F(v)\|^2.$$
The two terms in the left hand side of (\ref{pick P})) are positive so they are independently bounded by the right hand side.
The two independent bounds are 
\begin{equation} \label{bound f_j}
\sum_i \sum_{u \in P_i} \sum_{v \in N(u)} w_{uv} \chi(P(u) \neq P(v)) f_{j^{(i)}}(u)^2 \alpha _i^{-1} 
\end{equation}
$$ \leq 2C_1\tilde{\mathcal{R}}(F)^{-1/2} (1 + \sqrt{2}) \sum_{uv \in E^<} w_{uv} \|F(u) + F(v)\|^2$$
and
\begin{equation} \label{bound f_j + f_j}
\sum _i \sum_{u \in P_i} \sum_{v \in N(u)} (f_{j^{(i)}}(u) + f_{j^{(i)}}(v))^2\alpha_i^{-1} \leq 2(1 + \sqrt{2}) \sum_{uv \in E} w_{uv} \|F(u) + F(v)\|^2.
\end{equation}

Let $\hat{\alpha} = \max_i \alpha_i^{-1} f_{j^{(i)}}^2(u)$.
Choose $t \in (0,\hat{\alpha})$ uniformly and randomly and define two sets $S_{i,t} = \{u \in P_i : f_{j}(u) \geq \sqrt{t\alpha_i}\}$, $S_{i,t}' = \{u \in P_i : f_j(u) < - \sqrt{t\alpha_i}\}\}$.
Let $y_\ell = 1$ if $u_\ell \in S_{i,t}$, $y_\ell = -1$ if $u_\ell \in S_{i,t}'$, and $y_i = 0$ otherwise.
Note that $\sum_{v \in P_i}\sum_{u \in N(v)}w_{ij}|y_i + y_j| \geq \|\tilde{B}(S_t, S_t')\|$ and $\sum_{i\in V} |y_i|d(u_i) = \|C(S_t \cup S_t')\|$, so our theorem is equivalent to proving
$$\frac{\sum_{v \in P_i}\sum_{u \in N(v)}w_{ij}|y_i + y_j|}{\sum_{i \in V} |y_i|d(u_i)} \leq \frac{8C_1+4}{C_2(r - i + 1)} \sqrt{\tilde{\mathcal{R}}_G(F)}.$$

The expected volume of $S_{i,t} \cup S_{i,t}'$ is the mass of $P_i$: 
\begin{eqnarray*}
\mathbb{E}_t[\|C(S_{i,t} \cup S_{i,t}')\|] 			& = &	 \sum_{u \in P_i}d(u)\mathbb{P}[f_{j}(u)^2 \geq t\alpha _i] \\
						& = &    \sum _{u \in P_i}  d(u) f_{j}(u)^2 \alpha _i^{-1}\hat{\alpha}^{-1} \\
						& = & 	 \hat{\alpha}^{-1} \sum _{u \in P_i} \|F(u)\|^2 d(u) \\
						& \geq & \hat{\alpha}^{-1} C_2 \sum _{u \in V} \|F(u)\|^2 d(u).
\end{eqnarray*}

We claim that if $u,v \in P_i$, then 
$$ \mathbb{E}_{t} |y_i + y_j| \leq \hat{\alpha}^{-1} \alpha_i^{-1} |f_{j}(u) + f_{j}(v)|(|f_{j}(u)| + |f_{j}(v)|).$$
The proof of this splits into two cases: when $f_{j}(u)f_{j}(v) < 0$ and when $f_{j}(u)f_{j}(v) \geq 0$.
For the first case, assume that $f_{j}(u) < 0 < f_{j}(v)$.
We will only consider the case $|f_{j}(v)| \leq |f_{j}(u)|$; the other case follows similarly.
In this situation, the case becomes
\begin{eqnarray*}
\mathbb{E}_t |y_i + y_j|	& = &	|-1 + 1| \mathbb{P}\left(f_{j}(u) ^2, f_{j}(v) ^2 \geq t\alpha_i \right) + |0+0|\mathbb{P}\left(f_{j}(u) ^2, f_{j}(v) ^2 < t \alpha_i\right) \\
					& & \ \ + |-1+0|\mathbb{P}\left(f_{j}(v) ^2 < t \alpha_i \leq f_{j}(u) ^2\right) \\
				& = & \left|f_{j}(u) + f_{j}(v)\right|\left(|f_{j}(u)| + |f_{j}(v)|\right)\alpha _i^{-1} \hat{\alpha}^{-1}.
\end{eqnarray*}
For the second case, we have that $f_{j}(u)f_{j}(v) \geq 0$.
By symmetry, assume that $0 \leq f_{j}(u) ^2 \leq f_{j}(v) ^2$.
In this situation, the case becomes
\begin{eqnarray*}
\mathbb{E}_t |y_i + y_j|	&=& 	0 \cdot \mathbb{P}\left(f_{j}(u) ^2\leq f_{j}(v) ^2 < t \alpha _i\right) +  1 \cdot \mathbb{P}\left(f_{j}(u) < t\alpha_i \leq f_{j}(v) ^2\right) \\
					& & \ \ + 2\cdot \mathbb{P}\left(t \alpha_i \leq f_{j}(u) ^2\leq f_{j}(v) ^2 \right) \\
				& = & 	\left(|f_{j}(u)| + |f_{j}(v)|\right)\left|f_{j}(u) + f_{j}(v)\right| \alpha _i^{-1}\hat{\alpha}^{-1} .
\end{eqnarray*}
This concludes the proof to the claim.

We make use of this to count the bad stubs in all of our communities.

\newpage

\noindent Using (\ref{bound f_j}), 
\begin{eqnarray*}
\sum_i \mathbb{E}_t [\|\tilde{B}(S_{i,t}, S_{i,t}')\|]	& \leq &	\sum_i \sum_{u \in P_i} \sum_{v \in N(u)} w_{uv}\mathbb{P}[v \notin P_i, u \in S_t \cup S_t'] \\
					&   &	+ \sum_i \sum_{u \in P_i} \sum_{v \in N(u)} w_{uv}\mathbb{E}_t[|y_u + y_v| : u,v \in P_i] \\
					& \leq &\sum_i \sum_{u \in P_i} \sum_{v \in N(u)} w_{uv}\chi(P(u) \neq P(v)) f_j(u)^2 \alpha _i^{-1}\hat{\alpha}^{-1} \\
					&   &    +  \sum_i \sum_{u \in P_i} \sum_{v \in N(u)} w_{uv}\left(|f_{j}(u)| + |f_{j}(v)|\right)\left|f_{j}(u) + f_{j}(v)\right|\alpha _i^{-1}\hat{\alpha}^{-1}  \\
					& \leq & 2\hat{\alpha}^{-1}C_1(1 + \sqrt{2}) \tilde{\mathcal{R}}(F)^{-1/2} \sum_{ij \in E^<}  \|F(u) + F(v)\|^2 \\
					&   & + \hat{\alpha}^{-1} \sum_i \sum_{u \in P_i} \sum_{v \in N(u)} w_{uv}\left(|f_{j}(u)| + |f_{j}(v)|\right)\left|f_{j}(u) + f_{j}(v)\right|\alpha _i^{-1}.
\end{eqnarray*}
In order to bound the term 
$$Z = \sum_i \sum_{u \in P_i} \sum_{v \in N(u)} w_{uv}\left(|f_{j}(u)| + |f_{j}(v)|\right)\left|f_{j}(u) + f_{j}(v)\right|\alpha _i^{-1},$$
apply the Cauchy-Schwartz formula and (\ref{bound f_j + f_j}) as follows: 
\begin{eqnarray*}
\mathbb{E}_t [Z] 
					& \leq & \sqrt{ \sum_i \sum_{u \in P_i} \sum_{v \in N(u)} w_{uv}\alpha _i^{-1}  \left|f_{j^{(i)}}(u) + f_{j^{(i)}}(v)\right|^2} \\
					&     &   \cdot  \sqrt{ \sum_i \sum_{u \in P_i} \sum_{v \in N(u)} w_{uv}\alpha _i^{-1}  \left( |f_{j^{(i)}}(u)| + |f_{j^{(i)}}(v)|\right)^2} \\
					& \leq & \sqrt{ \sum_{uv \in E^<} 2(1 + \sqrt{2}) w_{uv}\|F(u) + F(v)\|^2} \sqrt{2\sum_{u \in V} d(u) \|F(u)\|^2} \\
					& = & 2\sqrt{\frac{(1 + \sqrt{2})}{\tilde{\mathcal{R}}(F)}}  \sum_{ij \in E^<}  \|F(u) + F(v)\|^2 .
\end{eqnarray*}
Plugging the bound on $Z$ into our previous bound yields
$$ \mathbb{E}_t \left[\sum_i \|\tilde{B}(S_{i,t}, S_{i,t}')\|\right] \leq 2\hat{\alpha}^{-1} \sqrt{\frac{(1 + \sqrt{2})}{\tilde{\mathcal{R}}(F)}} \left(C_1 \sqrt{1 + \sqrt{2}} + 1\right) \sum_{ij \in E^<}  \|F(u) + F(v)\|^2. $$

After $t$ is chosen, there will be some order such that $\|\tilde{B}(S_{i,t}, S_{i,t}')\| \leq \|\tilde{B}(S_{i',t}, S_{i',t}')\|$ when $i < i'$.
This ordering implies that 
$$\|\tilde{B}(S_{i,t}, S_{i,t}')\| \leq  \frac{1}{r+1-i} \sum_{i=1}^r \|\tilde{B}(S_{i,t}, S_{i,t}')\|. $$
Using $1 + \sqrt{2} < 3$ and $\sqrt{3} < 2$ we have that 
\begin{equation}\label{final ineq}
 \mathbb{E}_{t,P} \left[  \frac{\|C(S_{i,t} \cup S_{i,t}')\|}{C_2} \sqrt{\tilde{\mathcal{R}}(F)}  -   \frac{( r+1-i) \|\tilde{B}(S_{i,t}, S_{i,t}')\|}{8C_1 + 4} \right] > 0. 
\end{equation}
If $C(S_{i,t} \cup S_{i,t}') = \emptyset$, then $\tilde{B}(S_{i,t} \cup S_{i,t}') = \emptyset$ and the term inside \ref{final ineq} is zero.
So we may choose $t$ separately for each $i$ that performs at least as well as the expectation and satisfies $C(S_{i,t} \cup S_{i,t}') \neq \emptyset$.
\end{proof}

\subsection{Proof of Theorem \ref{main}.A}

Let $\Delta = (2\sqrt{k})^{-1}$. 
So each ball with diameter at most $\Delta$ contains at most $\frac{\mathcal{M}(V)}{k - 0.25}$ mass by Lemma \ref{spreading}.
Use Lemma \ref{partition happy} to partition $V$ into parts with with diameter at most $\Delta$, where for arbitrary edge $u,v$ we have that $ \mathbb{P}[P(u) \neq P(v)] \leq 4k d_M(u,v)$.
If two parts of the partition have mass less than $\frac{\mathcal{M}(V)}{2(k - 0.25)}$ each, then combine them (this process will maintain the property that each part has mass at most $\frac{\mathcal{M}(V)}{k - 0.25}$).

We claim that we now have at least $k$ parts with mass at least $\frac{\mathcal{M}(V)}{2(k - 0.25)}$.
If we have at most $k-1$ such parts in $P$, then the sum of the masses of those parts is at most $\frac{\mathcal{M}(V)(k-1)}{k - 0.25} = \mathcal{M}(V)\left(1 - \frac{3}{4(k - 0.25)}\right)$.
So either there are two parts left with mass at most $\frac{\mathcal{M}(V)}{2(k - 0.25)}$ that should have been combined, or there is one extra part with mass at least $\frac{\mathcal{M}(V)}{2(k - 0.25)}$.
This proves the claim.

The final step of the proof is to apply Theorem \ref{all the work} with $C_1 = 4k$, $r = k$, and $C_2 = \frac{1}{2(k - 0.25)}$.
\qed

\subsection{Proof of Theorem \ref{main}.B}

We follow the dimension reduction arguments of \cite{LGT}.
Let $h = 1200(2\ln(k) + \ln(200))$, and let $g_1, \ldots, g_h$ be random independent $k$-dimensional Gaussians, and define a projection $\Lambda:\mathbb{R}^k \rightarrow \mathbb{R}^h$ as $\Lambda(x) = h^{-1/2}(g_1^Tx, \ldots, g_h^Tx)$.
This mapping enjoys the properties (see \cite{LGT}) for any $x \in \mathbb{R}^k$ 
$$ \mathbb{E}[ \|\Lambda(x)\|^2 ] = \|x\|^2 $$
and 
$$ \mathbb{P}\left[ \|\Lambda(x)\| \notin [(1-\delta)\|x\|^2, (1+\delta)\|x\|^2] \right] \leq 2e^{-\delta^2h/12}. $$

Let $F^* = \Lambda \circ F$.
Recall Markov's inequality: if $X$ is a non-negative random variable, then $\mathbb{P}\left[\frac{X}{\mathbb{E}[X]} \geq a \right] \leq 1/a$.
This implies that with probability $0.9$ we have that 
$$\sum_{ij \in E^<}w_{ij}\|F^*(u_i) + F^*(u_j)\|^2 \leq 10\sum_{ij \in E^<}w_{ij}\|F(u_i) + F(u_j)\|^2.$$
Let $U_v = \{u \in V: \|F^*(u)\|^2 \in [0.9\|F(u)\|^2, 1.1\|F(u)\|^2]$.
We see that 
$$\mathbb{E}[ |V - U_v| ] \leq |V|2e^{-(0.1)^2h/12} = \frac{|V|}{100k^2}.$$
This implies that with probability $0.9$ we have that $\sum_{v \notin U_v} d(v) \|F(v)\|^2 \leq \frac{1}{10k^2}\sum_{v \notin V} d(v) \|F(v)\|^2$.
Therefore with the same $0.9$ probability we have that 
$$ \sum_{v \in V} d(v) \|F(v)^*\|^2 \geq \sum_{v \in U_v} d(v) \|F(v)^*\|^2 \geq \sum_{v \in U_v} d(v) 0.9 \|F(v)\|^2 \geq \sum_{v \in V} d(v) (0.9)^2 \|F(v)\|^2. $$
The probability of the intersection of two (possibly dependent) events, each with probability at least $0.9$, is at least $0.8$.
So with probability at least $0.8$ we have that 
$$\tilde{\mathcal{R}}_G(F^*)  \leq \frac{10}{(0.9)^2}\tilde{\mathcal{R}}_G(F).$$

We used $U_v$ to describe the set of vertices that ``behaved appropriately.''
We will now use $U_e$ to describe the set of pairs of vertices that ``behave appropriately.''
Let 
$$U_e = \{uv \in V^2: \|\Lambda \circ F'(u) - \Lambda \circ F_{vu}'(v)\| \in [0.9\|F'(u) - F_{vu}'(v)\|, 1.1\|F'(u) - F_{vu}'(v)\|]\}.$$
By definition, if $uv \in U_e$, then $d_M(F^*(u), F^*(v)) \in (1 \pm 0.1)d_M(u,v)$.  Observe,
\begin{eqnarray*}
\mathbb{E} \left[\sum_{uv \notin U_e} d(u)\|F(u)\|^2 d(v) \|F(v)\|^2 \right]
		& = & \sum_{uv \in V^2} d(u)\|F(u)\|^2 d(v) \|F(v)\|^2 \mathbb{P}[uv \notin U_e] \\
		& \leq & \sum_{uv \in V^2} d(u)\|F(u)\|^2 d(v) \|F(v)\|^2 2e^{-(0.1)^2h/12} \\
		& = & \left( \frac{\mathcal{M}(V)}{10k} \right)^2.
\end{eqnarray*}
We can then say that with $0.9$ probability we have 
$$\sum_{uv \notin U_e}d(u)d(v)\|F(u)\|^2\|F(v)\|^2 \leq 10\left( \frac{\mathcal{M}(V)}{10k} \right)^2 .$$
We claim that $F^*$ is spreading.
By way of contradiction, let $B$ be a ball in $\mathbb{R}^h$ with diameter at most $0.27$ and $\sum_{F^*(v) \in B} d(v) \|F(v)\|^2 > \frac{2\mathcal{M}(V)}{k}$.
Let $z$ be an arbitrary vertex such that $F^*(z) \in B$.
By the triangle inequality we have that $B_{0.27}(F^*(z)) \supset B$. 
Let $B' =  B_{0.3}(F(z))$, so that if $uz \in U_e$ and $F^*(u) \in B$, then $u \in B'$.
By Lemma \ref{spreading}, the mass of $B'$ is at most $\frac{\mathcal{M}(v)}{0.64k} \leq \frac{5\mathcal{M}(v)}{3k}$.
By our assumption, this implies that 
\begin{eqnarray*}
\sum_{vz \notin U_e} d(v) \|F(v)\|^2 
	& \geq & \sum_{F^*(v) \in B, F(v) \notin B'} d(v) \|F(v)\|^2 \\
	& \geq & \mathcal{M}(V)\left( \frac{2}{k} - \frac{5}{3k}\right)\\
	& = & \mathcal{M}(V)\frac{1}{3k}.
\end{eqnarray*}
We then sum this over all possible values of $z$ to get
\begin{eqnarray*}
10\left( \frac{\mathcal{M}(V)}{10k} \right)^2  & \geq & \sum_{uv \notin U_e}d(u)d(v)\|F(u)\|^2\|F(v)\|^2\\
		& \geq & \frac{1}{2} \sum_{F^*(z) \in B} d(z)\|F(z)\|^2 \sum_{vz \notin U_e} d(v) \|F(v)\|^2  \\
		& \geq & \frac{1}{2} \left(  \frac{2\mathcal{M}(v)}{k} \right) \left( \frac{\mathcal{M}(V)}{3k} \right).
\end{eqnarray*}
This is a contradiction, and therefore our claim is true.

The proof now easily follows from the proof to Theorem \ref{main}.A.
Project the points into $\mathbb{R}^h$, and then partition the points in the projected space using Lemma \ref{partition happy} and desired radius $\Delta = 0.27$.
We have probability $0.8$ that $\tilde{\mathcal{R}}_G(F^*)  \leq \frac{10}{(0.9)^2}\tilde{\mathcal{R}}_G(F)$ and probability $0.9$ that each ball of the projected space has mass at most $\frac{2\mathcal{M}(V)}{k}$, and so both of these things happen with probability at least $0.7$.
Similar to before, we may combine the parts of the partitions until we have at least $\frac{k}{2}$ parts, each with mass at least $\frac{\mathcal{M}(v)}{k}$.
The final step of the proof is to apply Theorem \ref{all the work} with $F^*$ instead of $F$, so that we may use $C_1 = \frac{2 \sqrt{1200\ln(200k^2)}}{0.27}$, $r = \frac{k}2$, and $C_2 = \frac{1}{k}$ to satisfy the assumptions of the theorem.
\qed

\section{Noisy Bipartite Hypercube} \label{NBH}
In this section we will give the details behind the noisy bipartite hypercube, Example \ref{bipartite noisy hypercubes}.
Let $k$ and $c$ be fixed, with $1 \leq c \leq \frac{10k}{22}$, and let $\epsilon = \frac{1}{\log_{2.2}(k/c)}$.
Let $G_{k,c}$ be the weighted complete graph on $2^k$ vertices, where each vertex corresponds to a finite binary sequence of length $k$ (in other words $V = \{0,1\}^k$), and the weight of edge $xy$ is $\epsilon^{\|x - y\|_1}$.
$G_{k,c}$ is called the \emph{noisy hypercube}.
Lee, Gharan, and Trevisan \cite{LGT} demonstrated a separation between the eigenvalues of $G_{k,c}$ and the conductance of small sets in the graph.

We define $G^{(o)}_{k,c}$ to be a complete bipartite spanning subgraph of $G_{k,c}$ such that $xy \in E(G^{(o)})$ (and keeps the same weight) if and only if $\|x - y\|_1$ is odd.
We will show that $G^{(o)}_{k,c}$ satisfies $2 - \lambda_{n - k} \leq 3 \epsilon$ and for any set $T,T' \subset V$ with $|T \cup T'| \leq \frac{c}{k}|V|$ we have that $\tilde{\phi}(T,T') \geq 1/2$.  This will show that Corollary \ref{main cor} is sharp.

The norm between $x,y\in \{0,1\}^k$ in the above example is defined to be the number of entries in which $x$ and $y$ are different (denoted by $\|x-y\|_1$).
We will drop the subscripts $k$ and $c$ from $G_{k,c}$ when it is clear.
For vertex subsets $A, B \subseteq V$ define $E(A,B)$ to be the set of edges with one endpoint in $A$ and the other endpoint in $B$; edges contained inside $A \cap B$ are counted twice.
Let $W(A,B)$ be the sum of the weights on the edges in $E(A,B)$.

\subsection{Background}
We can consider a vector $f$ as a map $f:V \rightarrow \mathbb{R}$, where $f(i)$ is the value in coordinate $i$ of the vector.
In this notation, we can think of the matrices $A$ and $L$ as operators on real-valued functions whose domain is $V$.
This notation - of maps and operators - also holds when we think of $V$ in terms of $\{0,1\}^k$ instead of $\{1,2,\ldots,n\}$.
For example, the \emph{adjacency matrix operator} is $A f(x)=\sum_{xy \in E} w_{xy}f(y)$.  
Let $\mathcal{H}_k$ denote the set of functions defined from $\{0,1\}^k$ into $\mathbb{R}$ with the inner-product of two functions $f,g\in\mathcal{H}_k$ defined by
$$\langle f,g\rangle =\frac{1}{n}\sum_{x\in V} f(x)g(x).$$
The \emph{$p$-norm} of a function $f$ is $\|f\|_p = \left(\frac{1}{n}\sum_{x\in V} |f(x)|^p\right)^{1/p}$, therefore $\|f\|_2 = \sqrt{\langle f,f \rangle}$.

Our proofs will make use of the rich field of study on maps whose domain is $\{0,1\}^k$.
Our notation follows that of \cite{W}.
We also found the course notes \cite{O notes} that O'Donnel grew into a book \cite{O book} to be enlightening.
We will need a select few theorems from this field, which we present below.

The \emph{Walsh functions} defined by 
$$W_S(x)=(-1)^{\sum_{i\in S} x_i}$$ for $S\subseteq [k]$ form an orthonormal basis for $\mathcal{H}_k$.
Thus, any $f\in\mathcal{H}_k$ can be written as $f(x)=\sum_{S\subseteq [k]} \widehat{f}(S)W_S(x)$ for some set of coefficients $\widehat{f}(S)$.
We call $\widehat{f}(S)$ the \emph{Fourier coefficients} of $f$ where
$$\widehat{f}(S) = \langle W_S, f \rangle = 2^{-k}\sum_{x\in V} f(x)(-1)^{\sum_{i\in S} x_i}.$$
Also recall Parseval's Identity which states that $\|f\|_2 ^2 = \sum_{S\subseteq [k]}\widehat{f}(S)^2$.

We define a \emph{noise process} $E_{\eta}$ to be a randomized automorphism on $\{0,1\}^k$, where $\mathbb{P}\left[E_{\eta}(x) = y\right] = \eta^{\|x-y\|_1}(1-\eta)^{k - \|x-y\|_1}$.
This is the standard model for independent bit-flip errors in coding theory.
The \emph{noise operator} is defined to be $N_{\eta} f(x) = \mathbb{E}\left[f(E_{\eta}(x))\right]$.
The noise operator is suggested to ``flatten out'' the values of $f$, although the exact strength to which this is true remains open \cite{BBBOW}. 
This process is intimately linked to Fourier coefficients and Walsh functions by $N_{\eta} f(x) = \sum_{S\subseteq [k]} \widehat{f}(S)\eta^{|S|}W_S(x)$ (this is also known as the Bonami-Beckner operator).
The final statement that we need is the Bonami-Beckner inequality: if $1 \leq p \leq q$ and $0 \leq \eta \leq \sqrt{(p-1)/(q-1)}$, then $\| N_{\eta} f \|_q \leq \|f\|_p$.
We will not need the full generality of this statement, just that if $0 \leq \eta \leq 1$, then 
\begin{equation}\label{Bonami-Beckner}
\sum_{S\subseteq [k]} \widehat{f}(S)^2 \eta^{2|S|}  \leq \|f\|_{1+\eta^2}^2.
\end{equation}

\subsection{Eigenvalues}

We begin by calculating the degree of a vertex in $G_{k,c}$.
Let $x$ be a fixed vertex, so that the degree of $x$ is 
$$\sum_{y \in V}\epsilon^{\|x - y\|_1} = \sum_{i=0}^k \epsilon^k \left|\{y : \|x - y\|_1 = i\}\right| = \sum_{i=0}^k {k \choose i}\epsilon^k = (1+\epsilon)^k.$$
Using this generating function, we see that the degree $d^o_k(x)$ of $x$ in $G^{(o)}_{k,c}$ is 
\begin{eqnarray*}
\frac{1}2 \left( (1+\epsilon)^k - (1-\epsilon)^k\right)	
		& = &	\frac{1}2 \left(\sum_{i=0}^k {k \choose i}\epsilon^i - \sum_{i=0}^k {k \choose i}(-\epsilon)^i \right) \\
		& = &   \frac{1}2 \sum_{i=0}^k {k \choose i} \epsilon^i \left(1 - (-1)^i \right) \\
		& = &   \sum_{0 \leq i \leq k, i\ \odd} {k \choose i} \epsilon^i \\
		& = & d^o_k(x).
\end{eqnarray*}

It will be convenient to define a graph $G^{(e)}$ to be the subgraph of $G_{k,c}$ where $E(G^{(e)}) = E(G_{k,c}) - E(G^{(o)}_{k,c})$.
Using a symmetrical argument, we see that each vertex in $G^{(e)}$ has degree $d^e_k(x) = \frac{1}2\left((1+\epsilon)^k + (1-\epsilon)^k\right)$.
We have chosen our $\epsilon$, $c$, and $k$ such that $d^o_k(x) \geq \frac{1}{2.2}(1 + \epsilon)^k$.

We will use the eigenvalues of the adjacency matrix to calculate the eigenvalues of the Laplacian of our graph.  Because our graph is regular, the eigenvectors of the Laplacian are the same as the eigenvectors of the adjacency matrix.  
We can use this information to directly calculate the eigenvalues associated to $G^{(o)}$. For this calculation we will need the eigenvalues of the normalized adjacency matrix.  The \emph{normalized adjacency matrix} is defined by $D^{-1/2} A D^{-1/2}$.  If $\rho$ is an eigenvalue of the adjacency matrix for a regular graph, $\overline{\rho}$ will denote the associated eigenvalue of the normalized version.

Let $\rho_S=\frac{(1+\epsilon)^{k-|S|}(1-\epsilon)^{|S|} - (1-\epsilon)^{k-|S|}(1+\epsilon)^{|S|}}{2}$.  The first lemma states that $W_S$ is an eigenfunction of the operator $A$ with eigenvalue $\rho_S$.

\begin{lemma}\label{eigenvalue}
Let $S\subseteq [k]$.  Then, $AW_S=\rho_S W_S$.
\end{lemma}

Using Lemma \ref{eigenvalue} we see that for each $0 \leq i \leq k$, with multiplicity ${k \choose i}$, the normalized Laplacian has eigenvalue 
\begin{equation}\label{adj eig}
\lambda_i = 1 - \overline{\rho_i} = 1 - \frac{(1+\epsilon)^{k-i}(1-\epsilon)^{i} - (1-\epsilon)^{k-i}(1+\epsilon)^{i}}{(1+\epsilon)^k - (1-\epsilon)^k}.
\end{equation}
By choosing the $k+1$ sets $S$ with $|S| \geq k-1$, we have that our eigenvalues satisfy $2 - \lambda_{n - k} \leq 2 \epsilon + \frac{(1-\epsilon)^{k-1}(1+\epsilon)}{(1+\epsilon)^k - (1-\epsilon)^k}$.
We used Mathematica to confirm that $\frac{(1-\epsilon)^{k-1}(1+\epsilon)}{(1+\epsilon)^k - (1-\epsilon)^k} \leq \epsilon$ for the ranges of $c,k$ allowed.

Now we return to give the proof of Lemma \ref{eigenvalue}.

\noindent\textit{Proof of Lemma \ref{eigenvalue}.}  Let $S\subseteq [k]$.  Consider the following:

\begin{small}\begin{eqnarray*}
A W_S(x) &=& \sum_{y\in N(x)}w_{xy} W_S(y)\\
	&=& \sum_{\substack{\left\|x-y\right\|_1\; \text{odd} \\  \sum_{i\in S}x_i = \sum_{i\in S} y_i (\text{mod}\;2)}}  \epsilon^{\left\| x-y\right\|_1} (-1)^{\sum_{i\in S}y_i} \\
	&+& \sum_{\substack{\left\|x-y\right\|_1\; \text{odd} \\  \sum_{i\in S}x_i \neq \sum_{i\in S} y_i (\text{mod}\;2)}}  \epsilon^{\left\| x-y\right\|_1} (-1)^{\sum_{i\in S}y_i}\\
	&=& W_S(x) \left( \sum_{\substack{\sum_{i\in S}x_i = \sum_{i\in S} y_i (\text{mod}\;2) \\ \sum_{i\in \overline{S}} x_i \neq \sum_{i\in \overline{S}} y_i (\text{mod}\; 2)}}  \epsilon^{\left\| x-y\right\|_1} -
 \sum_{\substack{\sum_{i\in S}x_i \neq \sum_{i\in S} y_i (\text{mod}\;2) \\ \sum_{i\in \overline{S}} x_i = \sum_{i\in \overline{S}} y_i (\text{mod}\; 2)}}  \epsilon^{\left\| x-y\right\|_1}\right).\\
\end{eqnarray*}\end{small}

First we will concentrate on the first summand in the above expression, call it $T_1$.  In $T_1$ we are summing over all $y\in N(x)$ such that the following two conditions hold:
\begin{enumerate}
\item $\sum_{i\in S}x_i = \sum_{i\in S} y_i (\text{mod}\;2)$ \\
\item $\sum_{i\in \overline{S}} x_i \neq \sum_{i\in \overline{S}} y_i (\text{mod}\; 2)$.
\end{enumerate}
Notice that $\left\| x-y\right\|_1 = \sum_{i\in S} |x_i-y_i| + \sum_{i\in \overline{S}} |x_i-y_i|$.  
Thus,
$$T_1=\left(\sum_{\substack{ \sum_{i\in S}x_i=\sum_{i\in S} y_i\; (\text{mod}\;2) \\ y_i\;\text{where}\; i\in S}} \epsilon^{\sum_{i\in S} |x_i-y_i| }\right)\left(\sum_{\substack{ \sum_{i\in \overline{S}}x_i\neq\sum_{i\in \overline{S}} y_i\; (\text{mod}\;2) \\ y_i\;\text{where}\; i\in \overline{S}}} \epsilon^{\sum_{i\in \overline{S}} |x_i-y_i| }\right).$$
In a similar fashion, the second summand, call it $T_2$, can be seen to be
$$\left(\sum_{\substack{ \sum_{i\in \overline{S}}x_i=\sum_{i\in \overline{S}} y_i\; (\text{mod}\;2) \\ y_i\;\text{where}\; i\in \overline{S}}} \epsilon^{\sum_{i\in \overline{S}} |x_i-y_i| }\right)\left(\sum_{\substack{ \sum_{i\in S}x_i\neq\sum_{i\in S} y_i\; (\text{mod}\;2) \\ y_i\;\text{where}\; i\in S}} \epsilon^{\sum_{i\in S} |x_i-y_i| }\right).$$

Now, notice that $T_1=d_{|S|}^e\cdot d_{|\overline{S}|}^o$ and $T_2=d_{|\overline{S}|}^e\cdot d_{|S|}^o$.  Pulling everything back together we observe

$$A W_S(x)= \left(d_{|S|}^e\cdot d_{|\overline{S}|}^o - d_{|\overline{S}|}^e\cdot d_{|S|}^o \right) W_S(x) = \rho_S W_S(x). \hfill\qed$$  

Our proof about small sets having large conductance will make use of the fact that 
\begin{eqnarray*}
\overline{\rho_{|S|}} &=& \frac{(1+\epsilon)^{k-|S|}(1-\epsilon)^{|S|} - (1-\epsilon)^{k-|S|}(1+\epsilon)^{|S|}}{(1+\epsilon)^k - (1-\epsilon)^k}\\
 &\leq& \frac{11}{10}\left(\frac{(1+\epsilon)^{k-|S|}(1-\epsilon)^{|S|}}{(1+\epsilon)^k}\right)\\
 &=& \frac{11}{10}\left( \frac{1-\epsilon}{1+\epsilon}\right)^{|S|}.
\end{eqnarray*}

\subsection{Conductance}
In this subsection we will prove that $\phi (T)\geq \frac{1}{2}$ for any $T\subset V$ with $\left|T\right|<\frac{c}{k}\left|V\right|=\frac{c}{k}n$.  
Recall that $\phi(T)=\frac{w(T,\overline{T})}{w(T,V)}$ and let $T=T'\cup T''$.  
Because $\tilde{\phi}(T',T'') = \phi(T' \cup T'') + \frac{w(T',T') + w(T'', T'')}{w(T' \cup T'',V)}$, this will conclude the details of Example \ref{bipartite noisy hypercubes}.
We will require the following two lemmas.

\begin{lemma}
Let $T\subseteq [n]$ and define $\mathbbm{1}_T\in\mathcal{H}$ to be the characteristic function of $T$.  
Under these conditions,
$$\frac{1}{d_k^0} \langle \mathbbm{1}_T, A\mathbbm{1}_T \rangle  \leq \frac{11}{10}\sum_{S\subseteq [k]} \left(\frac{1-\epsilon}{1+\epsilon}\right)^{|S|}\left( \widehat{\mathbbm{1}_T}(S)\right)^2.$$
\end{lemma}

\noindent\textit{Proof.}  
Let $T\subseteq [n]$.  
Recall that $A W_S = \rho_S W_S$, and so by (\ref{adj eig}) we have that $\frac{1}{d_k^o} AW_S = \overline{\rho_{|S|}} W_S \leq \frac{11}{10}\left( \frac{1-\epsilon}{1+\epsilon}\right)^{|S|}W_S$.
Using that the Walsh functions form an orthonormal basis, we observe:
\begin{eqnarray*}
\frac{1}{d_k^0} \langle \mathbbm{1}_T, A \mathbbm{1}_T\rangle 
			&=& \frac{1}{n} \sum_{x\in \{0,1\}^k} \mathbbm{1}_T(x) \frac{1}{d_k^o} A\mathbbm{1}_T(x)\\
			&=& \frac{1}{n}\sum_{x\in\{0,1\}^k} \left( \sum_{S\subseteq [k]} \widehat{\mathbbm{1}_T}(S)W_S(x)\right) \left( \sum_{S'\subseteq [k]} \widehat{\mathbbm{1}_T}(S')\frac{1}{d_k^o}AW_{S'}(x)\right)\\
			&=& \frac{1}{n}\sum_{x\in\{0,1\}^k} \left( \sum_{S\subseteq [k]} \widehat{\mathbbm{1}_T}(S)W_S(x)\right) \left( \sum_{S'\subseteq [k]} \widehat{\mathbbm{1}_T}(S')\overline{\rho_{|S'|}}W_{S'}(x)\right)\\
			&=& \sum_{S\subseteq [k]}\sum_{S'\subseteq [k]}\widehat{\mathbbm{1}_T}(S)\overline{\rho_{|S'|}} \widehat{\mathbbm{1}_T}(S')\left(\frac{1}{n} \sum_{x\in\{0,1\}^k}W_S(x)W_{S'}(x)\right)\\
			&=& \sum_{S\subseteq [k]} \overline{\rho_{|S|}} \left( \widehat{\mathbbm{1}_T}(S)\right)^2 \\
			&\leq& \sum_{S\subseteq [k]} \frac{11}{10}\left( \frac{1-\epsilon}{1+\epsilon}\right)^{|S|} \left( \widehat{\mathbbm{1}_T}(S)\right)^2 . \hfill\qed
\end{eqnarray*}

\begin{lemma}
Let $T\subseteq [n]$.  Then, 
$$w(T,\overline{T}) = w(T,V)-  n \langle \mathbbm{1}_T, A \mathbbm{1}_T\rangle.$$
\end{lemma}

\noindent\textit{Proof.}  
Let $T\subseteq [n]$.  
Notice that what we really are proving is that $n \langle \mathbbm{1}_T, A \mathbbm{1}_T\rangle = w(T,T)$.  
Now consider

\begin{eqnarray*}
n\langle \mathbbm{1}_T, A\mathbbm{1}_T\rangle &=& \sum_{x\in\{0,1\}^k} \mathbbm{1}_T(x)A\mathbbm{1}_T(x)\\
&=& \sum_{x\in\{0,1\}^k} \mathbbm{1}_T(x) \left(\sum_{y\in N(x)} w(x,y) \mathbbm{1}_T(y)\right)\\
&=& \sum_{x \in T}\sum_{y \in T} w_{xy}\\
&=& w(T,T). \hfill\qed
\end{eqnarray*}

We are now ready to prove the main result from this section.

\begin{theorem}
The conductance $\phi (T)\geq \frac{1}{2}$ for any $T\subset V$ with $\left|T\right|<\frac{c}{k}\left|V\right|=\frac{c}{k}n$.
\end{theorem}

\noindent\textit{Proof.}  
Let $T\subseteq [n]$ be such that $|T|\leq\frac{c}{k}n$.  
Then, 
$$\phi(T) = \frac{w(T,\overline{T})}{w(T,V)}= 1-\frac{n}{|T|d^o_k} \langle \mathbbm{1}_T, A\mathbbm{1}_T\rangle.$$

By (\ref{Bonami-Beckner}) with $\eta = \sqrt{\frac{1-\epsilon}{1+\epsilon}}$, we have that 
\begin{eqnarray*}
\frac{n}{|T|d^o_k} \langle \mathbbm{1}_T, A\mathbbm{1}_T\rangle 
	& \leq & \frac{n}{|T|} \frac{11}{10}\sum_{S\subseteq [k]} \left(\frac{1-\epsilon}{1+\epsilon}\right)^{|S|}\left( \widehat{\mathbbm{1}_T}(S)\right)^2 \\
	& \leq & \frac{11n}{10|T|} \left\|\mathbbm{1}_T\right\|^2_{1+\eta^2}\\
	& = & \frac{11n}{10|T|} \left(\frac{|T|}{n}\right)^{2/(1+\eta^2)} \\
	& = & \frac{11}{10} \left( \frac{|T|}{n} \right)^\epsilon.
\end{eqnarray*} 

By choice of $\epsilon$, we have that 
\begin{eqnarray*}
\phi(T) &\geq & 1-\frac{11}{10} \left( \frac{|T|}{n} \right)^\epsilon \\
	& \geq & 1 - \frac{11}{10} \left( \frac{c}{k} \right)^\epsilon \\
 	& \geq & 1 - \frac{1}{2} \qed
\end{eqnarray*}

\section{Empirical Tests} \label{empirical section}

\subsection{An algorithm}
We do not recommend trying to implement the argument in Section \ref{NBH} for real-world use.
The construction was optimized for rigorous bounds at the cost of efficiency and real-world performance.
We now present a modified version of the construction in the proof.
The outline of the construction is the same.
This modified version does not have any rigorous bounds, but it has good performance and does not require significant computational power.
We also take advantage of things observed by applied mathematicians.
For example, the theoretical proof partitions the points in spectral space greedily, which gives poor but rigorous bounds on concentration.
However, there is ample empirical evidence \cite{PSSMF} that the spectral space of real-world graphs are strongly clusterable.
Also, when we project down to one dimension, we do not necessarily project down onto one axis.
By examining the $(V_{17}, V_{18})$ plot in Figure $2(a)$ of \cite{PSSMF}, we see that some communities are best detected using a combination of several eigenvectors.

For each $x \in C_i'$, we define $x_{c'} = x$ if $\|x_{c'} - c'\| \leq \|x_{c'} + c'\|$ and $x_{c'} = -x$ otherwise.
We define $d'(x, y) = \min\{\|x-y\|_2, \|x + y\|_2\} = d_M(x,y)$.

\bigskip

\noindent \textbf{Pseudo-Code:} Inputs: $k$, $G$, $r$, $t$, $t'$.  Outputs: $S_1, S_1', S_2, S_2', \ldots, S_r, S_r'$.
\begin{enumerate}
	\item Calculate $F$, and throw out all points at the origin.
	\item Let $F'(u) = F(u)/\|F(u)\|$ and let $X \subset \mathbb{R}^k$ be the range of $F'$.
	\item Calculate $r$ random centers $c_i^{(0)}$ such that $\|c_i^{(0)}\|=1$.
	\item Run $r$-means.  For $j=1\ldots,t'$:
	\begin{enumerate}
		\item Initialize each cluster $C_i = \emptyset$ for $1 \leq i \leq r$.
		\item For each point $x \in X$, 
		\begin{enumerate}
			\item Find the value of $i^*$ such that $d'(x, c_{i^*}^{(j-1)})$ is minimum.  
			\item If $d'(x, c_{i^*}^{(j-1)}) < 2^{-1/2}$, then assign $x \rightarrow C_{i^*}$, otherwise set $x \rightarrow C_{r+1}$.
		\end{enumerate}
		\item Calculate the centers: for $1 \leq i \leq r$, 
		\begin{enumerate}
			\item If $C_{i}$ is nonempty, calculate $c_{i}^{(j)} = \frac{\sum_{x \in C_i}\mathcal{M}(x)x_{c_i^{(j-1)}}}{\|\sum_{x \in C_i}\mathcal{M}(x)x_{c_i^{(j-1)}}\|}$.
			\item If $C_i = \emptyset$, set $c_{i}^{(j)}$ to equal a random point $x \in C_{r+1}$.  Leave $C_i$ empty.
		\end{enumerate}
		\item Repeat (b) and (c) as necessary.
	\end{enumerate}
	\item For $i=1,\ldots,r$, if $C_i$ is non-empty do:
	\begin{enumerate}
		\item For each vertex $u\in C_i$, calculate $z_u = F(u)^T c_{i}^{(t')}$.
		\item Find appropriate $n^*<0$ and $p^*>0$ as thresholds.
		\item Set $S_i = \{u \in C_i : z_u \geq p^*\}$ and $S_i' = \{v \in C_i : z_v \leq n^*\}$.
	\end{enumerate}
\end{enumerate}

\subsection{Results}
Because finding the largest eigenvectors is an approximate algorithm, we will abuse notation by saying that a vector $v \in \mathbb{R}^k$ is ``at the origin'' if $\|v\| < 10^{-8}$.
Since $e_i$ denotes an eigenvector, we will use the notation $a_i$ to denote the unit vector that is $1$ in the $i^{th}$ coordinate and $0$ in all other coordinates.
When we say $x \approx 10a_2 - 12a_7$, we have represented the vector $x$ using an approximation by deleting any $a_i$ whose coefficient is less than $5$.
To find appropriate values for $n^*$ and $p^*$, we tested every pair of values under two conditions:
\begin{enumerate}
	\item when $2 \leq |S_i|,|S_i'|\leq30$, or
	\item $2 \leq |S_i|,|S_i'| \leq 3000$ and $n^* = p^*$
\end{enumerate}
 and used the pair of values that produced the smallest bipartite conductance.

Recall that the bipartite conductance from our $i^{\mbox{th}}$ strongest community must be at least $2 - \lambda_{n+1-i}$.
We will use this to compare our communities to the ``best possible'' based on the eigenvalues that we calculate.
It has been commented that the ``best'' theoretical bounds \cite{PM,LRTV} for community detection use linear programming for a continuous relaxation instead of spectral methods.
The best bounds from linear programming are $O(\sqrt{\log(n)})$.
In the graphs we encountered, the spectral values are not very small, and therefore the bounds from the eigenvalues performed much better in practice.

Our algorithm was run on four data sets below: a biological network, the hyperlink structure between a set of websites, a traffic routing network for telecommunication companies, and relationships between fictional characters.
Our heuristic algorithm found bipartite communities where the $i^{\mbox{th}}$ best community has bipartite conductance less than $10(2-\lambda_i)$.
This is significantly better than the bound in Theorem \ref{main}, and borders on the best possible.
Despite that bipartite conductance is larger than conductance on the same vertex set, the second and third best bipartite communities found by our algorithm had lower bipartite conductance than the conductance of the second and third best classical communities found by a standard algorithm in the telecommunications network!

Our algorithm found communities with relevant structure in all but the biological network.
On political blogs, our algorithm found the Authority/Hub framework first described by Kleinberg \cite{K}.
On telecommunication networks, our algorithm found a community local to a regional network (Korea) rather than the dense formation at the logical center.
Furthermore, the two sets of the community provided information about the peering relationship.
This can be used to infer the \emph{level} of a telecommunications company, which approximates how close it is to the logical center of the Internet.
Information about levels can be used to efficiently route traffic \cite{BPK} by idealizing the network as a hyperbolic space.
Hence our results do not just score well; they have qualitative significance too.

\bigskip

\noindent \textbf{Double Mutant Combinations}

Costanzo et. al. \cite{Minnesota2} prepared a data set of how a colony of yeast would react when a pair of genes were deleted, which is available at the supplementary online material website \cite{Minnesota Data}.
This is the data set discussed in Section 1 when the difference between a bicluster and a bipartite community is clarified.
A yeast colony typically grows at a rate of $t$, and when gene $i$ is mutated it grows at rate $\delta_i t$.
The double mutant combinations is then an analysis of when genes $i$ and $j$ are deleted and the yeast colony grows at a rate of $(\delta_i\delta_j + \epsilon_{ij})t$.
We specifically worked with data set $S4$, where edge $ij$ exists if the experimental value of $|\epsilon_{ij}|$ is more than $0.08$, and the $p$-value for the true value of $\epsilon_{ij}$ equaling $0$ is less than $0.05$.
We chose this specific data set because it was recommended to us by one of the authors, Chad Myers.

The experiment specifically only tested gene combinations with one gene from an array set $A$ and the other gene from a query set $Q$.
Both sets are large, with $|A| = 3885$ and $|Q| = 1711$.
We used the graph induced by the intersection of the two lists, where $|A \cap Q| = 1139$.
This induced subgraph has $33821$ directed edges.
There were $1719$ edges that were close to the cut-off threshold and were only represented in one direction.
We chose to include an undirected edge if either orientation of it exists in the directed graph; this produced $17770$ undirected edges without multiplicities.

We originally ran our algorithm with $20$ eigenvectors.
However, the expected distance between two random unit vectors in $\mathbb{R}^k$ is $\approx \sqrt{1 - \frac{1}{k^2}}$.
In the end this space was too sparse, and none of the parts of the partitions contained more than $9$ genes.
We modified our algorithm to run on $6$ eigenvectors, and we only required the radius for each partition to be $\sqrt{2}$ instead of $\frac{1}{\sqrt{2}}$.
This change was unique to this data set.  The basic facts about the eigenvalues can be seen in the table below.

\[\begin{array}{c|c|c|c}
i	&2 - \lambda_{n+1-i} 		& \# (f_i(u) > 10^{-4})	 & \# (-f_i(u) > 10^{-4}) \\
\hline
1	&0.584549		&	523			&	519			\\				
2	&0.595037		&	656			&	464			\\									
3	&0.605708		&	662			&	455			\\
4	&0.62266		&	584			&	536			\\
5	&0.633406		&	550			&	562			\\	 				
6	&0.644139		&	587			&	529								
\end{array}\]

There was one vertex at the origin that was thrown out.
As you can see, this graph has no good bipartite communities.
We ran $r$-means to find three clusters.
In every case, $n^* = p^* = 0$ (and so $|S_i| = \#(z_j > 0)$ and $|S_i'| = \#(z_j < 0)$).

\[ \begin{array}{c|c|c|c|c}
\mbox{i} 	& 100c_i	 					& \#(z_j > 0)		& \#(z_j < 0)	&\tilde{\phi}_G(S_i, S_i') 			\\
\hline
1   		& -5a_1 + 48a_2 - 53a_3 - 13a_4 -13a_5 - 66a_6		& 166			& 171		&0.740774		\\
2		& 72a_1 + 54a_2 + 43a_3 - 8a_4 - 10a_5			& 231 			& 122		&0.702708		\\
3		& 7a_1 - 23a_2 - 96a_4 + 5a_5				& 179 			& 269		&0.656735		
\end{array}\]

The behavior is clear: this graph has no bipartite communities and so the algorithm is spreading the communities out to try to cover the entire graph.
Recall that we originally established that communities are vaguely defined terms.
Conductance is only one measure of a ``good community.''
Our algorithm has proven that this is not the correct method for this data set.
Conductance is a measure that wants the community to be exclusive, while the \emph{modular hypothesis} \cite{Minnesota} suggests that each module may be in many communities with other modules.
Hence their algorithm to find bipartite communities only counts good edges, while not significantly penalizing for bad edges.
Based on our discussion in Section 1, it may be better for this application to use $L'$ instead of $L$.

There is some silver lining to this result - it does fix one of the trade offs mentioned in the Discussion section of \cite{Minnesota} .
Their methods had very strict conditions for when a set of vertices formed a bipartite community.
Those conditions led to very small communities: they reported a mode of $11$ genes per community, and it appears that none have more than $110$ genes (see Figure S5 in supplementary materials).
These communities do not cover entire pathways, and an ad-hoc procedure is developed to reduce overlapping communities into one single subset of a pathway.
Our algorithm naturally looked for larger communities, and each likely contains entire pathway(s).

\bigskip

\noindent \textbf{Political Blogs}

We have a graph with $1490$ blogs that focus on political matters.
This data set was originally put together by Adamic for a paper by Adamic and Glance \cite{AG}; we found it on a list of data sets maintained by Mark Newman \cite{N}.
The name of each blog is given, and a value is given for whether the blog has a liberal or conservative bias.
There were $758$ blogs with a liberal bias and $732$ with a conservative bias.
The graph contains an unweighted directed multi-edge from blog $a$ to blog $b$ for each time blog $a$ contains a hyperlink to blog $b$.
We turn this into a weighted undirected simple graph, where the weight on edge $ab$ is the number of directed edges in the original graph from $a$ to $b$ or from $b$ to $a$.

The normalized Laplacian of our new graph has $2$ as an eigenvalue with multiplicity $1$.
The maximum eigenvector is nonzero in just two coordinates, at blogs ``digital-democrat.blogspot.com'' and ``thelonedem.com,'' each of which is in just one edge with weight $2$ to the other.
We call this the \emph{trivial bipartite community}.
The graph also has $266$ isolated vertices (blogs that never linked or were linked by other blogs).
We remove the isolated blogs and the blogs in the trivial community, and continue on the reduced graph.
We call our reduced graph the \emph{blog graph}.
We present the basic facts about the eigenvalues below.

\[\begin{array}{c|c|c|c}
i	&2 - \lambda_{n+1-i} 		& \# (f_i(u) > 10^{-4})	 & \# (-f_i(u) > 10^{-4}) \\
\hline
1	&0.235069		&	7			&	4			\\				
2	&0.286025		&	101			&	70			\\									
3	&0.289204		&	55			&	54			\\
4	&0.367392		&	184			&	190			\\
5	&0.404912		&	304			&	808			\\	 				
6	&0.412217		&	605			&	394								
\end{array}\]

There were no vertices at the origin that were thrown out.
We ran $r$-means to find three clusters.

\[ \begin{array}{c|c|c|c|c|c|c|c|c}
\mbox{i} 	& 100c_i	 			& \#(z_j > 0)		& \#(z_j < 0)	& p^*		&|S_i|		&n^*		&|S_i'|		&\tilde{\phi}_G(S_i, S_i') 			\\
\hline
1   		& 6a_4 + 88a_5 + 47a_6			& 97 			& 260		&0.111		&2		&-0.0659	&9		&0.535		\\
2		& -47a_5 + 88a_6			& 447 			& 156		&3.67(10^{-5})	&447		&-4.83(10^{-5})	&156		&0.484		\\
3		& -100a_2				& 27 			& 10		&0.00283	&7		&-0.0137	&2		&0.771			
\end{array}\]

The communities found by our algorithm are somewhat strong, with $\tilde{\phi}_G(S_i, S_i') \leq 3 (2 - \lambda_i)$.
We will assess our algorithms ability to pass the ``eye-test'' by finding expected structures inside our reported communities.

Based on previous applications of bipartite communities mentioned in Section 1, we know of two possible structures that might be expected in our communities.
One is that we might see group-versus-group antagonistic behavior (called a \emph{flame war}), which would be represented by many links between blogs from different political parties.
The second structure is the Authority-Hub framework suggested by HITS, which would be represented by a uniform orientation of the original links.
Another sign of the Authority-Hub framework is that one side should have a large in-degree and the other side should have a large out-degree.
We will now define a few parameters that will help us assess whether or not these structures are present in our results.
Let $\FLAME$ denote the ratio of edges that involve a blog from each political party among all edges that cross from $S_i$ to $S_i'$.
Let $d^+$ denote the average out-degree and $d^-$ denote the average in-degree based on the hyperlink orientations in the original data set.
Finally, let $H_{\mbox{value}}$ denote ratio of edges, among all edges that cross from $S_i$ to $S_i'$, that are oriented from a blog with a positive projection score to a blog with a negative projection score.  
Because this would be $0.5$ in a random graph, we set $H_{\mbox{score}} = 4*(0.5 - H_{\mbox{value}})^2 \in [0,1]$.
By this construction, a large $H_{\mbox{score}}$ would indicate a strong Authority/Hub structure without bias from which of $S_i$ and $S_i'$ is the set of Hubs and which is the set of Authorities.

First, we calculate these structural properties for the whole cluster $C_i$, before we calculate $n^*$ and $p^*$.
\[ \begin{array}{c|c|c|c||c|c||c|c}
\mbox{i}	&\FLAME		&H_{\mbox{value}}	&H_{\mbox{score}}		& z_j > 0	&				& z_j < 0	&	\\
	 	& 		& 			& 				& d^+		& d^- 				& d^+		& d^- \\
\hline
1		&0.061		&0.193775		&0.375095			&18.4433	&40.866				&14.2731	&6.4	\\
2		&0.045		&0.857111		&0.510113			&14.3289	&4.6868				& 15.0064	&41.8782	\\
3		&0.038		&0.435897		&0.0164366			&15.8519	&14.037				& 26.6		&15.8	
\end{array}\]

The first conclusion is that this algorithm did not pick up even a trace of a flame war.
Cluster 2, and to a lesser extent Cluster 1, do demonstrate an Authority-Hub framework.
Now we see how these parameters adjust when we restrict $C_i$ to $(S_i, S_i')$.
\[ \begin{array}{c|c|c|c||c|c||c|c}
\mbox{i}	&\FLAME		&H_{\mbox{value}}	&H_{\mbox{score}}		& S_i		&				& S_i'		&	\\
	 	& 		& 			& 				& d^+		& d^- 				& d^+		& d^- \\
\hline
1		&0		&0.9			&0.64				&10		&1.5				&0.888889	&1.33333	\\
2		&0.044		&0.857111		&0.510113			&14.3289	&4.6868				& 15.0064	&41.8782	\\
3		&0.167		&0.333333		&0.111111			&4.14286	&2				& 21.5		&9.5	
\end{array}\]

The parameters for the second community do not change because $S_2 \cup S_2' = C_2$.  
The first bipartite community now displays a very strong Authority/Hub structure, but the roles have reversed.

\bigskip

\noindent \textbf{Autonomous Systems}

Our next data set is the CAIDA relationships dataset for Autonomous Systems (AS) from November 12, 2007, which was downloaded from the Stanford SNAP project \cite{SNAP}, who received it from Leskovec, Kleinberg, and Faloutsos \cite{LKF}.
We associated each Autonomous System identifier with a name using a table found at Geoff Huston's personal website \cite{H}.
An autonomous system is a communications company that routes Internet traffic.  
The data represents a collection of inter-company connections used for routing traffic through several in-between carriers.
From the information provided, we created an unweighted undirected graph.
Our dataset also includes information about the type of relationship (customer, provider, or peer) that two linked companies have, which we choose to ignore until we perform an autopsy on our results.

The graph contains AS's $1$ through $65535$.
However, as one AS buys another, or some AS disappears for any other reason, only half of the AS's in that range were active at the time our graph was made.
Specifically, $39146$ of those addresses were not in any relationships, and so we removed them.
We clustered using the top twenty eigenvalues, none of which had trivial eigenvectors.
We describe the results below.

\[\begin{array}{c|c|c|c}
i	&2 - \lambda_{n+1-i} 		& \# (f_i(u) > 10^{-4}) & \# (f_i(u) < -10^{-4})\\
\hline
1	& 0.0112			& 40		& 	96				\\
2	&0.0376		 		& 295 		& 	657				\\
3	& 0.0397			& 432		& 	253				\\
4	&0.0562		 		& 1860 		& 	2693				\\ 
5	&	0.0598			&182		& 	339				\\
6	&0.0643				& 3322 		& 	8148				\\
7	& 	0.0644			&  1147 	& 	 1635				\\
8	&	0.0648	 		&1834 		& 	 1723				\\
9	&	0.0656			& 2043 		& 	1340				\\
10	&0.0661		 		& 1015 		& 	841				\\
11	&0.0673				&2070 		& 	1002				\\
12	&0.0697				&660		& 	1261				\\
13	&0.0703				&964 		& 	1201				\\
14	&0.0725				&1556 		& 	3634				\\
15	&0.0738				&900		& 	1741				\\
16	&	0.0784			& 755		& 	967				\\
17	&	0.0801			&226		& 	 780				\\
18	&	0.0804			&1935 		& 	1533				\\
19	&	0.0804			&1702 		& 	2728				\\
20	&	0.0816			&1031 		&  1083
\end{array}\]

\[ \begin{array}{l|l}
\mbox{i} 	& 100c_i \\
\hline
1   		& 8a_6 	-6a_8 + 93a_{14} + 22a_{15} - 11a_{18} - 21a_{19}	\\
2		& 66a_4 - 48a_6 + 16a_7 + 34a_8 - 10a_{10} + 20a_{11} + 8a_{12} + 14a_{13} \\
			& - 7a_{14} + 7a_{17} + 8a_{18} + 30a_{19}		\\
3		& 	-29a_3 -7a_6-95a_{13} 		\\
4		& 	7a_2 -94a_{14} + 34a_{15} 	\\
5		& 	15a_4 + 7a_5 - 68a_6 + 20a_7 + 17a_8 - 6a_{10} + 20a_{11} + 26a_{13} \\
			& - 48a_{14} - 19a_{16} + 23a_{19}	\\
6		& 	7a_5 - 87a_6 + 9a_7 - 8a_8 + 20a_{13} + 36a_{14} + 11a_{15} \\
			& - 16a_{18} - 9a_{19} - 5a_{20}	\\
7		& 	-5a_3 - 15a_4 - 9a_5 + 86a_6 - 15a_7 - 13a_8 - 11a_{11} - 31a_{13} \\
			&+ 13a_{14} + 8a_{16} + 6a_{18} - 20a_{19}		\\
8		& 	18a_2 + 33a_4 + 7a_5 - 30a_6 + 7a_7 + 35a_8 + 22a_{9} - 13a_{10} + 8a_{11} \\
			& + 9a_{13} - 62a_{14} + 23a_{15} + 5a_{17} + 9a_{18} + 32a_{19} + 5a_{20}	\\
9		& 	-94a_6 + 10a_7 - 7a_8 + 7a_{13} - 24a_{18} + 17a_{19} 	\\
10		& 	15a_4 + 11a_5 - 55a_6 + 32a_7 + 20a_8 - 6a_{10} + 39a_{11} + 24a_{13} \\
			& - 24a_{14} - 45a_{16} + 18a_{19} + 11a_{20}				
\end{array}\]

\[ \begin{array}{c|c|c|c|c|c|c|c}
\mbox{i} 	&	 \#(z_j > 0)		& \#(z_j < 0)	& p^*			&|S_i|		&n^*			&|S_i'|		&\tilde{\phi}_G(S_i, S_i') 			\\
\hline
1		&  	413			& 631		&0.00118457		&117 		&-0.00123966		&	34	&0.296417				\\
2		& 	1598 			&	581	&2.93973(10^{-5})	&1593 		&-2.5241(10^{-5})	&571		&0.529198				\\
3		&	277			& 188		&0.0133517		&154 		&-0.0114952		&15 		&0.0920771				\\
4		&	240			&  204		&0.0328121		&5 		&-0.0344949		&53 		&0.117318				\\
5		&	4892			& 609		&0.00041631		&3 		&-0.000375232		&2 		&	0.333333				\\
6		&	155			&	487	&0.000342043		&13		&-0.00036664		&102		&0.152				\\
7		&	737			&	2283	&0.000480553		&2 		&-0.000514521		&2 		&0.333333				\\
8		&	285			& 304		&0.000273189		&34 		&-0.00027188		&147 		&	0.218018			\\		
9		&	840			& 945		&0.0110021		&39 		&-0.00784598		&532		&	0.06838			\\
10		&	1299			& 136		&0.000264905		&5		&-0.00028486		&6		&0.294118
\end{array}\]

The algorithm returned at least one weak community ($\tilde{\phi}_G(S_2, S_2') > 0.5$) and several trivial communities ($|S_5 \cup S_5'| = 5$ and $|S_7 \cup S_7'| = 4$).
However, the seven other communities are within $7$ times the best possible, and some of them are within $4$ times the best possible.
As a comparison, we also ran classical community detection algorithms on the graph.
We made only two changes to our heuristic algorithm to do this: we calculated the smallest eigenpairs of $L$ and used standard distance Euclidean distance instead of $d'$.
Note that if we skip step 5 and just returned $C_i$ as the community, then this algorithm would be equivalent to the one developed by Ng, Jordan, and Weiss \cite{NJW}.
To avoid confusion, the set of vertices forming the $i^{\mbox{th}}$ classical community will be denoted $S_i^*$.
As a fair comparison, we also tested for the $20$ smallest eigenvalues and clustered for the best $10$ clusters.

\[\begin{array}{c|c|c|c}
i	&\lambda_{i}	 		& \# (f_i(u) > 10^{-4}) & \# (f_i(u) < -10^{-4})\\
\hline
1	&0			&26389		& 	0				\\
2	&0.0112			& 184		& 	738				\\
3	&0.0178			&3313 		& 	22998				\\
4	&0.0195			&1569 		& 	24508				\\
5	&0.0222			& 24731 	& 	1290				\\
6	&0.0267			&22781 		&  	3017				\\
7	& 0.0297		& 22530 	& 	 3427				\\
8	&0.0359	 		& 18197 	& 	 2383				\\
9	& 0.0399		&1012 		& 	2644			\\
10	&0.0416	 		& 6204  	& 	19671			\\ 
11	&0.0430			&17607		& 	4359				\\
12	&0.0433			& 21238		& 	2861			\\
13	& 0.044			&  815	 	& 	2214			\\
14	&0.0446	 		&20494		& 	 2755				\\
15	&0.0512			& 18730 	& 	2530				\\
16	&0.0577	 		& 5370	 	& 	19274			\\
17	&0.0591			&4025 		& 	11216			\\
18	&0.0605			&4448 		& 	13195				\\
19	&0.0627			&7134 		& 	17049			\\
20	&0.0636			&18480 		& 	4614				\\
\end{array}\]

\[ \begin{array}{l|l}
\mbox{i} 	& 100c_i \\
\hline
1   		& 86a_1 - 15a_3 - 7a_4 + 13a_5 + 11a_6 + 26a_7 + 6a_8 + 10a_{10} \\
			& + 12a_{14} + 22a_{16} - 16a_{19} - 14a_{20}	\\
2		& 86a_1 - 9a_3 - 10a_4 + 8a_7 -21a_{10} - 14a_{16} - 12a_{17} \\
			& + 38a_{19} + 6a_{20} \\
3		& 14a_1 + 12a_7 + 5a_8 + 92 a_{10} - 12a_{11} - 16a_{12} \\
			& - 6a_{13} - 24a_{14} - 7a_{16} \\
4		& 49a_1 - 11a_3 + 10a_5 + 10a_6 + 23a_7 + 6a_8 + 10a_{10} + 5a_{12} \\
			& + 16a_{14} + 27a_{16} - 6a_{17} - 9a_{18} - 63a_{19} - 38a_{20}\\
5		& 94a_1 - 12a_3 - 9a_4 +9 a_5 + 7a_6  + 14a_7 - 16a_{10} + 7a_{14} \\
			& - 9a_{16} - 7a_{17} + 8a_{19} \\
6		& 31a_1 - 6a_3 +6 a_5 +6a_6 + 11a_7 + 9a_{10} + 7a_{14} + 62a_{16} \\
			& + 10a_{18} -25a_{19} - 63a_{20} \\
7		& 74a_1 - 15a_3 - 7a_4 + 13a_5 + 11a_6 + 28a_7- 23a_{10} +8a_{11} + 9a_{12}  \\
			&+ 23a_{14} + 6a_{15} - 26a_{16} - 8a_{17} - 10a_{18} - 24a_{19} + 22a_{20}\\
8		& 8a_1 -6a_8 -6a_{10} - 12a_{14}- 6a_{16} - 6a_{17} + 98a_{19} - 6a_{20}\\
9		& 91a_1 - 14a_3 - 7a_4 + 13a_5 + 10a_6 + 24a_7 - 14a_{10} + 6a_{11}  \\
			&+ 7a_{12} + 14a_{14} + 5a_{15} - 9a_{16} -8a_{19}\\
10		& 85a_1 - 16a_3 + -6a_4 + 14a_5	+ 12a_6 + 27a_7 + 5a_8 - 17a_{10} + 8a_{11} \\
			& + 10a_{12} + 19a_{14} + 7a_{15} - 13a_{16} - 6a_{18} - 14a_{19} + 5a_{20}
\end{array}\]

\[ \begin{array}{c|c|c|c|c|c|c|c}
\mbox{i} 	&	 |C_i|		& p^*			&|S_i ^*|		&\phi_G(S_i^*) 			\\
\hline
1		& 	 1690		& 0.00327541		&1605 		&0.225106			\\
2		& 	 1430		&0.00353641		&661		&0.24767				\\
3		&	597		&0.00870002		&592		&0.0534351			\\
4		&	528		&0.00469703		&441		&0.133382				\\
5		&	4390		&0.00325309		&2986		&0.43507				\\
6		&	200		&0.00515704		&180		&0.122349				\\
7		&	729		&0.0035845		&729 		&0.131891				\\
8		&	302		&0.0107398		&237		&0.101093		\\		
9		&	3197		&0.0033048		&2964		&0.600235			\\
10		&	8045		&0.00371243		&755		&0.445483
\end{array}\]

\begin{figure}
\begin{center}
\scalebox{0.18}{\includegraphics{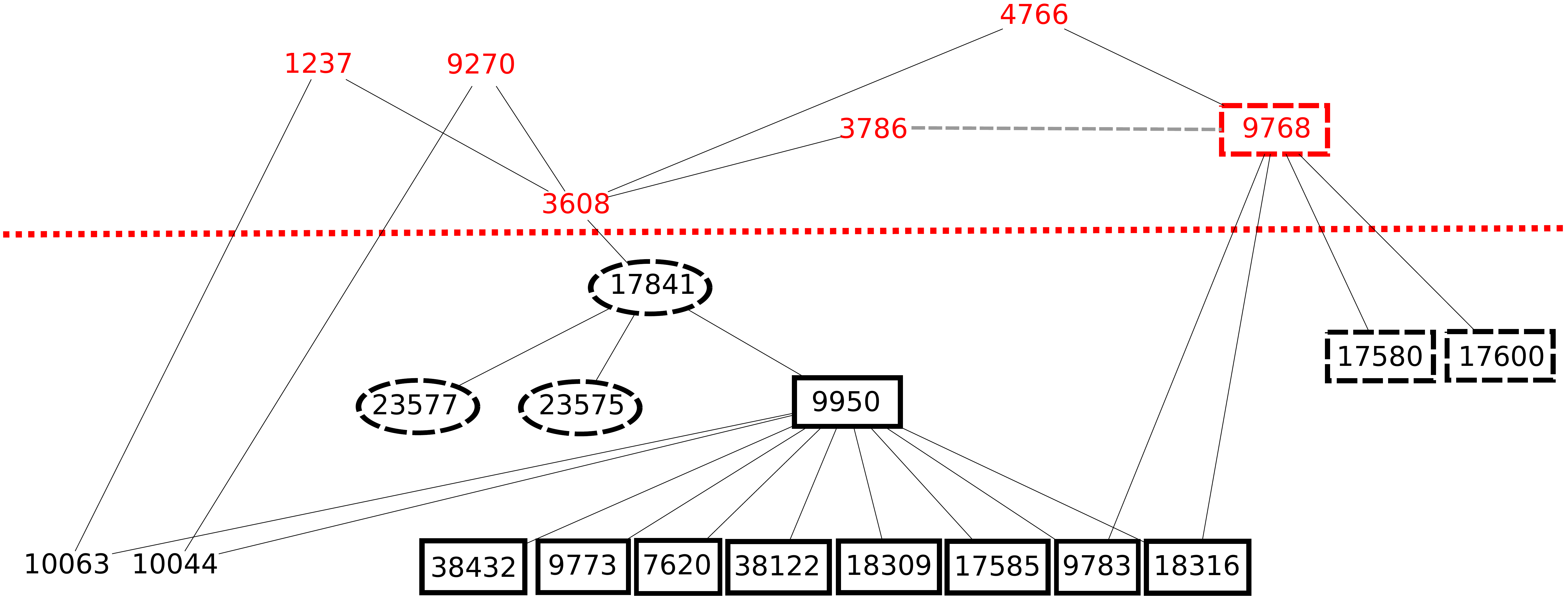}}
\caption{Dashed circles represent top members of $S_3^*$, dashed boxes represent top members of $S_9 \cup S_9'$, and solid boxes are top members of both.  ASN's below the dotted line have all of their relationships included in the diagram, each of the ASN's above the dotted line have at least one relationship not depicted here.  The diagram uses verticality to represent peering relationships, as lower ASN's are customers of higher ASN's.  The unique horizontal dashed line represents a peer-to-peer relationship.}
\label{ASN figure}
\end{center}
\end{figure}

We can immediately see that the continuous relaxation has stronger solutions for classical communities than for bipartite communities.
Specifically, there are six non-trivial eigenvalues less than $0.03$, while only one eigenvalue is at least $1.97$.
The classical algorithm also had no issue with trivial communities, as the smallest community returned with $180$ members.
However, the classical algorithm had more issues with weak communities than the bipartite algorithm; as it had three communities with classical conductance over $0.4$ compared to one community with bipartite conductance over $0.4$, and one community with classical conductance over $0.6$ compared to no communities with bipartite conductance that large.
The strongest communities from the two algorithms are quite comparable: the best classical community has a stronger score than the best bipartite community, the second and third best bipartite communities have stronger scores than the second and third best classical communities respectively, and the fourth best classical community is better than the fourth best bipartite community.

The communities discovered by the two algorithms are largely disjoint, with the notable exception of the best-scoring communities from each algorithm.
The twelve AS's in $S_3^*$ with the best $z_i$ score are in order: 
$$7620, 9773, 17585, 18308, 38122, 38432, 9950,   9783,    18316,  23575, 23577, \mbox{and } 17841.$$
The top two best scorers in $S_9$ are $9950$ and $9768$, while the top ten in $S_9'$ are:
$$18308, 17585, 38122, 38432, 7620, 9773, 9783, 18316, 17580, \mbox{and } 17600.$$
All of the AS's listed above are based in Korea.
A diagram of the connections between these AS's is presented in Figure \ref{ASN figure}.
The diagram demonstrates that the difference between $S_9$ and $S_9'$ contains information about peering relationships.

\bigskip

\noindent \textbf{MARVEL Characters}

We have a graph with $6486$ characters and $12941$ journals owned by publisher MARVEL.
This data was put together by Alberich, Miro-Julia, and Rosell\'{o} \cite{AMR}, and we found it on the Amazon Web Services \cite{AMAZON} list of large data sets.
The graph is bipartite; a character is linked to a journal title if the character appears in that journal.
From this, we create a different undirected unweighted graph.
Each vertex corresponds to a character, and the two characters are adjacent if there exists a journal that they both appear in.
We call our new graph the \emph{MARVEL graph}.

Among the largest eleven eigenvalues of the normalized Laplacian of the MARVEL graph, there are eigenvalues $2$ and $1.5$, with multiplicities $1$ and $9$, respectively.
It is well known that the multiplicity of $2$ as an eigenvalue corresponds to the number of bipartite components in the graph.
In this case, the MARVEL graph has one bipartite component, and it is one edge between the characters ``MASTER OF VENGEANCE'' and ``STEEL SPIDER/OLLIE O.''
The space of eigenvectors with eigenvalue $1.5$ can be generated by vectors $(v_1, \ldots, v_9)$, where each $v_i$ is non-zero in exactly two coordinates.
Furthermore, each non-zero coordinate of $v_i$ corresponds to a vertex with degree two, and the vertices are adjacent.  
As an odd structural motif, each of the $9$ pairs of vertices have a common neighbor.
We call these ten bipartite communities the trivial communities of the MARVEL graph.
We display information for the ten largest eigenvalues for the non-trivial communities below.
\[\begin{array}{c|c|c|c}
i	&2 - \lambda_{n+1-i} 		& \# (f_i(u) > 10^{-4})	 & \# (-f_i(u) > 10^{-4}) \\
\hline
1	&0.419864		&	36			&	19			\\				
2	&0.532353 		&	24			&	39			\\									
3	&0.551096		&	28			&	83			\\
4	&0.558939		&	25			&	40			\\
5	&0.60149		&	49			&	49			\\	 				
6	&0.606019		&	60			&	46			\\					
7	&0.625912		&	293			&	1132			\\					
8	&0.638269		&	558			&	161			\\					
9	&0.650475		&	4468			&	1503			\\					
10	&0.65465		&	571			&	2886			\\					
\end{array}\]

Using the above ten eigenvectors, we threw out $17$ vertices at the origin in addition to the deleted trivial communities.
We found four clusters among the remaining vertices.
The centers are dominated by just a few of the eigenvectors, and those eigenvectors are the ones with many non-zero coordinates.
The basic stats of the clusters are listed below.

\[ \begin{array}{c|c|c|c|c|c|c|c|c}
\mbox{i} 	& 100c_i	 			& \#(z_j > 0)		& \#(z_j < 0)	& p^*		&|S_i|		&n^*		&|S_i'|		&\tilde{\phi}_G(S_i, S_i') 			\\
\hline
1   		& -5a_7 + 99a_9 - 12a_{10}		& 3987 			& 1294		&0.000820	&2586		&-0.000822	&411		&0.553		\\
2		& 31a_8	+ 94a_9 - 12a_{10}		& 448 			& 145		&0.000217	&395		&-0.000239	&90		&0.717			\\
3		& -a_7					& 82 			& 29		&0.0489		&3		&-0.0901	&2		&0.697			\\
4		& 12a_9 + 99a_{10}			& 62			& 76 		&0.0204		&5		&-0.0150	&16		&0.498	
\end{array}\]

By examining the eigenvalues, we conclude from this that the MARVEL graph simply does not have strong bipartite communities.
However, our algorithm did find bipartite communities with bipartite conductance that is within $4-30\%$ of best possible.

The point of using the MARVEL graph instead of the ubiquitous Hollywood graph is that we can de-anonymize the nodes and use an ``eye-test'' to see if the bipartite communities have any significance.
The descriptions of some of the characters in our bipartite communities are accessible by a quick internet search; some of the characters are too obscure to find their background.
Based on the characters whose backgrounds we were able to track down, our communities do have a cohesive theme.
Most of the top scorers in $S_1 \cup S_1'$ have a scientific or pseudo-scientific background (``ZABO,''	``PAST MASTER,'' ``DR. JOANNE TUMOLO,'' and ``DR. EDWIN HAWKINS''); characters in $S_1$ are villains and characters in $S_1'$ are side characters. 
Most of the top scorers in $S_2 \cup S_2'$ are from the ``Spiderman'' comics.
Two of the top five scorers in $S_2$ are villains (``BRAINSTORM'' and ``ROCKET RACER II''), and two others are minor characters (``SARAH CHAN'' and ``CLARICE BERNHARD'').
On the other side, the second and fourth highest scorers in $S_2'$ are Spiderman's wife and boss (``MARY WATSON-PARKER'' and ``J. JONAH JAMESON'').
All of the top scorers in $S_3 \cup S_3'$ involve the comic series surrounding the protagonist ``Dr. Strange.''
Furthermore, the top scorers in $S_3$ are different manifestations of Dr. Strange (``DR VINCENT STEVENS,'' ``STRANGE,'' ``NOBLE,'' and ``PARADOX'').
The characters in $S_3'$ include a villain (``SISTER NIL''), a love interest (``CLEA''), and a financial relationship (``AZOPARDI'').
The classical community formed by $S_4 \cup S_4'$ is centered on a setting called ``EARTH-9910,'' but we found no clear distinction between the characters in $S_4$ and the characters in $S_4'$.

\noindent \textbf{Acknowledgments.}  We would like to thank Geoffrey Sanders and Noah Streib for their helpful notes on this manuscript.  
We would also like to thank Randall Dahlberg for his support and assistance.
We are grateful for Shiping Liu's kind emails that filled in the holes of the literature review from earlier drafts of this manuscript.
Finally, we would like to thank Chad Myers for his help in finding and understanding the gene interaction data set.



\begin{thebibliography}{1}
\bibitem{AG} L. Adamic and N. Glance, ``The Political Blogosphere and the 2004 U.S. Election: Divided They Blog.'' \textit{Proceedings of the 3rd LinkDD} (2005) 36--43.

\bibitem{AL} F. Atay and S. Liu, ``Cheeger constants, structural balance, and spectral clustering analysis for signed graphs.'' \textit{MPI MiS Preprint 111/2014}.

\bibitem{AMR} R. Alberich, J. Miro-Julia, and F. Rossell\'{o}, ``Marvel Universe Looks Almost Like a Real Social Network.'' \textit{arXiv} $cond-mat/0202174v1$

\bibitem{AMAZON} Amazon Web Services' Public Data Set \url{aws.amazon.com/publicdatasets/}

\bibitem{BBBOW} K. Ball, F. Barte, W. Bednorz, K. Oleszkiewicz, and P. Wolff, ``$L^1$ smoothing for the Ornstein-Uhlenbeck semigroup.'' \textit{Mathematika} \textbf{59}, 1 (2013), 160--168.

\bibitem{BJ} F. Bauer and J. Jost, ``Bipartite and Neighborhood Graphs and the Spectrum of the Normalized Graph Laplace Operator.'' \textit{Communications in Analysis and Geometry} \textbf{21}, 4 (2013), 787 -- 845.

\bibitem{Minnesota} J. Bellay, G. Atluri, T. Sing, K. Toufighi, M. Costanzo, P Ribeiro, G. Pandey, J. Baller, B. VanderSluis, M. Michaut, S. Han, P. Kim, G. Brown, B. Andrews, C. Boone, V. Kumar, and C. Myers,
		``Putting genetic interactions in context through a global modular decomposition.'' \textit{Genome Research} \textit{21} (2011), 1375 -- 1387.

\bibitem{BPK} M. Bogu\~{n}\'a, F. Papadopoulos, and D. Krioukov, ``Sustaining the Internet with hyperbolic mapping.'' \textit{Nature Communications} \textbf{1} (2010) .

\bibitem{CCGGP} M. Charikar, C. Chekuri, A. Goel, S. Guha, and S. Plotkin, ``Approximating a Finite Metric by a Small Number of Tree Metrics.'' \textit{39th Symposium on Foundations of Computer Science} (1998), 379 -- 388.

\bibitem{C} F. Chung, ``Four Cheeger-type Inequalities for Graph Partitioning Algorithms.'' \textit{Proceedings of ICCM}, II (2007), 751--772.

\bibitem{CFSV} V. Colizza, A. Flammini, M. Serrano, and A. Vespignani, ``Detecting Rich-club Ordering in Complex Networks.'' \textit{Nature Physics} \textbf{2} (2006), 110--115.

\bibitem{Minnesota2} M. Costanzo, A. Baryshnikova, J. Bellay, Y. Kim, E. Spear, C. Sevier, H. Ding, J. Koh, K. Toufighi, S. Mostafavi, J. Prinz, R. St Onge, B. VanderSluis, T. Makhnevych, F. Vizeacoumar, S. Alizadeh, S. Bahr, R. Brost, Y. Chen, M. Cokol, R. Deshpande, Z. Li, Z. Lin, W. Liang, M. Marback, J. Paw, B. San Luis, E. Shuteriqi, A. Tong, N. van Dyk, I. Wallace, J. Whitney, M. Weirauch, G. Zhong, H. Zhu, W. Houry, M. Brudno, S. Ragibizadeh, B. Papp, C. Pal, F. Roth, G. Giaever, C. Nislow, O. Troyanskaya, H. Bussey, G. Bader, A. Gingras, Q. Morris, P. Kim, C. Kaiser, C. Myers, B. Andrews, and C. Boone, ``The Genetic Landscape of a Cell.'' \textit{Science} \textbf{327} (2010), 425--431. 

\bibitem{Minnesota Data} M. Costanzo, A. Baryshnikova, J. Bellay, Y. Kim, E. Spear, C. Sevier, H. Ding, J. Koh, K. Toufighi, S. Mostafavi, J. Prinz, R. St Onge, B. VanderSluis, T. Makhnevych, F. Vizeacoumar, S. Alizadeh, S. Bahr, R. Brost, Y. Chen, M. Cokol, R. Deshpande, Z. Li, Z. Lin, W. Liang, M. Marback, J. Paw, B. San Luis, E. Shuteriqi, A. Tong, N. van Dyk, I. Wallace, J. Whitney, M. Weirauch, G. Zhong, H. Zhu, W. Houry, M. Brudno, S. Ragibizadeh, B. Papp, C. Pal, F. Roth, G. Giaever, C. Nislow, O. Troyanskaya, H. Bussey, G. Bader, A. Gingras, Q. Morris, P. Kim, C. Kaiser, C. Myers, B. Andrews, and C. Boone, \url{http://drygin.ccbr.utoronto.ca/~costanzo2009/}

\bibitem{DP} C. Delorme and S. Poljak, ``The Performance of an Eigenvalue Bound on the Max-cut Problem in Some Classes of Graphs.'' \textit{Discrete Mathematics} \textbf{111} (1993), 145--156.

\bibitem{F} S. Fortunato, ``Community Detection in Graphs.'' \textit{arXiv} $0906.0612v2$.

\bibitem{G} G. Gallier, ``Spectral Graph Theory of Unsigned and Signed Graphs Applications to Graph Clustering: a Survey'' (2015 manuscript available at \url{http://www.cis.upenn.edu/~jean/hot.html}, which is an update from the 2013 manuscript \textit{arXiv} $1311.2492$)


\bibitem{GKL} A. Gupta, R. Krauthgamer, and J. Lee, ``Bounded Geometries, Fractals, and Low-distortion Embeddings.'' \textit{44th Symposium on Foundations of Computer Science} (2003), 534--543.

\bibitem{HST} V. Henson, G. Sanders, and J. Trask, ``Extremal Eigenpairs of Adjacency Matrices Wear Their Sleeves Near Their Hearts: Maximum Principles and Decay Rates for Resolving Community Structure.'' \textit{Lawrence Livermore National Laboratory Technical Report} \textbf{LLNL-TR-618872} (available at \url{https://library-ext.llnl.gov}).

\bibitem{H} G. Huston, \url{bgp.potaroo.net/cidr/autnums.html} 

\bibitem{K} J. Kleinberg, ``Authoritative Sources in a Hyperlinked Environment.'' \textit{Journal of the ACM} \textbf{46}, 5 (1999), 604--632.

\bibitem{KBCG} Y. Kluger, R. Basri, J. Chang, and M. Gerstein, ``Spectral Biclustering of Microarray Data: Coclustering Genes and Conditions.'' \textit{Genome Research} \textbf{13} (2003), 703 -- 716.


\bibitem{KSLLLA} J. Kunegis, S. Schmidt, A. Lommatzsch, J. Lerner,  E. De Luca and S. Albayrak, ``Spectral analysis of signed graphs for clustering, prediction and visualization.'' \textit{Proc. SIAM Int. Conf. on Data Mining} \textbf{12} (2010) 559--570.


\bibitem{LGT} J. Lee, S. Gharan, and L. Trevisan, ``Multi-way spectral partitioning and higher-order Cheeger inequalities.'' \textit{Proceedings of the 44th ACM STOC} (2012), 1117 -- 1130.

\bibitem{LN} J. Lee and A. Naor, ``Extending Lipschitz Functions via Random Metric Partitions.'' \textit{Inventiones Mathematicae} \textbf{160} (2005), 59--95.

\bibitem{LKF} J. Leskovec, J. Kleinberg and C. Faloutsos, ``Graphs over Time: Densification Laws, Shrinking Diameters and Possible Explanations.'' \textit{Proceedings of the 11th ACM SIGKDD} (2005), 177--187.

\bibitem{LLLW} J. Li, G. Liu, H. Li, and L. Wong, ``Maximal Biclique Subgraphs and Closed Pattern Pairs of the Adjacency Matrix: A One-to-One Correspondence and Mining Algorithms.'' \textit{IEEE Transactions on Knowledge and Data Engineering} \textbf{19}, 12 (2007), 1625--1637. 

\bibitem{L} S. Liu, ``Multi-way dual Cheeger constants and spectral bounds of graphs.'' \textit{Advances in Mathematics} \textbf{268} (2015), 306--338.

\bibitem{LRTV} A. Louis, P. Raghavendra, P. Tetali, and S. Vempala, ``Many Sparse Cuts via Higher Eigenvalues.''  \textit{Proceedings of the 44th ACM STOC} (2012), 1131 -- 1140.

\bibitem{N} M. Newman, \url{www-personal.umich.edu/~mejn/netdata/}

\bibitem{NJW} A. Ng, M. Jordan, and Y. Weiss, ``On Spectral Clustering: Analysis and an Algorithm.'' \textit{Advances in Neual Information Processing Systems}, MIT Press (2001), 849 -- 856.

\bibitem{SNAP} Stanford Large Network Dataset Collection \url{http://snap.stanford.edu/data/index.html} 

\bibitem{LSPZL} D. Lo, D. Surian, P. Prasetyo, K. Zhang, and E. Lim, ``Mining Direct Antagonistic Communities in Signed Social Networks.'' \textit{Information Processing and Management} \textbf{49}, 4 (2013), 773--791.

\bibitem{MP} B. Mohar and S. Poljak, ``Eigenvalues and the Max-cut Problem.'' \textit{Czechoslovak Mathematical Journal}, \textbf{40}, 2 (1990), 343--352.

\bibitem{O notes} R. O'Donnel, Lectures 15 and 16, \textit{course notes scribed by Ryan Williams and Ryan O'Donnel}, \url{http://www.cs.cmu.edu/~odonnell/boolean-analysis/}

\bibitem{O book} R. O'Donnel, Analysis of Boolean Functions, \textit{Cambridge University Press} (2014).


\bibitem{PM} G. Puleo and O. Milenkovic, ``Community Detection via Minimax Correlation Clustering and Biclustering.'' \textit{arXiv} $1506.08189v1$.

\bibitem{PSSMF} B. Prakash, M. Seshadri, A. Sridharan, S. Machiraju, and C. Falostsos, ``EigenSpokes: Surprising Patterns and Scalable Community Chipping in Large Graphs.'' \textit{2009 IEEE International Conference on Data Mining Workshops} 290--295.

\bibitem{T} L. Trevisan, ``Max Cut and the Smallest Eigenvalue.''  \textit{SIAM J. Computing} \textbf{41}, 6 (2012), 1769--1786.

\bibitem{TBGGT} C. Tsourakakis, F. Bonchi, A. Gionis, F. Gullo, and M. Tsiarli, ``Denser than the Densest Subgraph: Extracting Optimal Quasi-cliques with Quality Guarantees.'' \textit{Proceedings of the 19th ACM SIGKDD} (2013), 104--112.

\bibitem{VM} D. Verma and M. Meil{\u a}, ``A Comparison of Clustering Algorithms.''  Technical Report \textbf{03-05-01}.

\bibitem{W}  R. de Wolf, ``A Brief Introduction to Fourier Analysis on the Boolean Cube.'' \textit{Theory of Computing Library Graduate Surveys} \textbf{1} (2008), 1 -- 20.

\bibitem{WYWLZ} L. Wu, X. Ying, X. Wu, A. Lu, and Z-H. Zhou, ``Spectral Analysis of k-balanced Signed Graphs.'' \textit{Proceedings of the 15th Pacific-Asia Conference on Knowledge Discovery and Data Mining (PAKDD11)} (2011) 1--12.

\bibitem{ZLL} K. Zhang, D. Lo, and E. Lim, ``Mining Antagonistic Communities from Social Netowrks.''  \textit{Proceedings of the 14th PAKDD} \textbf{1} (2010), 68--80.

\end{thebibliography}
\end{document}